\documentclass[11pt,a4paper,reqno]{amsart}
\usepackage{graphicx} 
\usepackage{amssymb}
\usepackage{amsthm}
\usepackage{booktabs}
\usepackage{pdflscape}
\usepackage[a4paper,left=2.5cm,right=2.5cm,top=2.8cm,bottom=2.5cm]{geometry}

\usepackage[foot]{amsaddr}

\usepackage[numbers,sort&compress]{natbib}
\bibpunct[, ]{[}{]}{,}{n}{,}{,}
\renewcommand\bibfont{\fontsize{10}{12}\selectfont}

\usepackage{amsmath} 
\usepackage{caption}
\usepackage{algorithmic}
\usepackage[ruled,linesnumbered,vlined]{algorithm2e}
\usepackage{hyperref}
\usepackage{multirow}
\usepackage{subcaption}
\newdimen\figrasterwd
\figrasterwd\textwidth
\usepackage[shortlabels]{enumitem}
\usepackage{xcolor}
\usepackage{cleveref}
\usepackage{enumitem}
\newcommand{\mscLink}[1]{#1}

\newtheorem{theorem}{Theorem}[section]
\newtheorem{lemma}[theorem]{Lemma}
\newtheorem{corollary}[theorem]{Corollary}
\newtheorem{proposition}[theorem]{Proposition}

\newtheorem{definition}[theorem]{Definition}

\newtheorem{assumption}{Assumption}

\allowdisplaybreaks
\usepackage{tikz}
\usetikzlibrary{calc}
\usepackage{pgfplots}
\pgfplotsset{compat=newest}
\usetikzlibrary{arrows,arrows.meta,intersections,shapes.geometric,positioning,calc,patterns}

\begin{document}
\title[A constraint dissolving method for problems with geometric constraints]{A constraint dissolving inexact penalty method for optimization problems with geometric constraints}




\author{Xiaoxi Jia$^1$}\address{$^1$Faculty of Mathematics, Chemnitz University of Technology, Germany\\ Mail: \textsc{\{xiaoxi.jia,manuel.schaller\}@math.tu-chemnitz.de}}
\author{Leander Lerch$^2$}\address{$^2$Faculty of Electrical Engineering and Information Technology, Chemnitz University of Technology, Germany\\ Mail: Mail: \textsc{\{leander.lerch,stefan.streif\}@etit.tu-chemnitz.de}}
\author{Stefan Streif$^2$}
\author{Manuel Schaller$^1$}

\thanks{The authors acknowledge funding by the European Social Fund Plus (ESF Plus) and the Free State of Saxony through the project MORE-KIBA}

\maketitle
{
\small\textbf{\abstractname.}
Optimization problems with geometric constraints have a broad range of applications, including machine learning, finance, and control. A powerful algorithmic tool to resolve these geometric constraints are constraint dissolving methods.
To this end, we propose a framework for constraint dissolving mappings for nonconvex geometric constraints. 
Leveraging these, we develop a constraint dissolving inexact penalty method to solve 
optimization problems with general set-membership constraints and possibly nonconvex geometric constraints. 
We establish the convergence of the proposed algorithm and prove that every feasible accumulation point is Mordukhovich stationary. Notably, we rely only on mild asymptotic Mordukhovich regularity, which is significantly weaker than the constraint qualifications adopted in the existing literature on constraint dissolving methods. Numerical experiments addressing classical equality-, complementarity-, sparsity-, and low-rank constrained optimization problems demonstrate that the proposed method is competitive with the safeguarded augmented Lagrangian method in terms of solution quality and significantly outperforms the penalty decomposition method.
\par\addvspace{\baselineskip}

{
\small\textbf{Keywords.}
	Geometric constraints $\cdot$ Constraint dissolving methods $\cdot$ Penalty-type methods $\cdot$ Asymptotic regularity 
\par\addvspace{\baselineskip}
}

{
\small\textbf{AMS subject classifications.}
	\mscLink{49J52}, \mscLink{90C26}, \mscLink{90C30}
\par\addvspace{\baselineskip}
}

\maketitle

\section{Introduction}
\noindent Besides classical equality and inequality constraints, geometric constraints, characterized by their intrinsic geometric structures, have attracted considerable attention in recent years due to their broad applications in machine learning, signal processing, control, finance, and many other areas. Typical examples include sparsity constraints, complementarity constraints, switching constraints, low-rank constraints, and manifold constraints. 

As a prototype for these applications, we will focus on the optimization problem
\begin{equation}\label{Eq:1}\tag{P}
\begin{aligned}
  \min \ &f(x)  \\ 
  \text{s.t.} \ & G(x)\in C, \ x\in D,
\end{aligned}
\end{equation}
where $f:\mathbb R^n \to \mathbb R$ and $G: \mathbb R^n\to \mathbb R^m$ are continuously differentiable, $C \subset \mathbb R^m$ is a nonempty, closed, and convex set, while $D\subset \mathbb R^n$ is assumed to be nonempty and closed, but not necessarily convex. Owing to the potential nonconvexity of $D$ and its ability to model a broad class of constraint sets, we refer to the constraint $x\in D$
as a \emph{geometric constraint}. 

To the best of our knowledge, (inexact) penalty-based methods have become one of the most successful approaches for 
optimization methods with geometric constraints.
Existing methods can be broadly classified into three categories: augmented Lagrangian methods \cite{de2023constrained, jia2023augmented}, penalty decomposition methods \cite{kanzow2023inexact}, and constraint dissolving methods which however only have been studied for convex problems~\cite{xiao2024dissolving,xiao2026exact}. Extending them to nonconvex geometric constraints constitutes the main focus of this manuscript.

A safeguard augmented Lagrangian method for solving \eqref{Eq:1} was introduced in \cite{jia2023augmented}, where the disjunctive, sparsity, and low-rank constraints were considered. To the best of our knowledge, this was the first algorithmic framework capable of handling the general formulation \eqref{Eq:1} with the arbitrary closed set $D$ and the convex set $C$, rather than being restricted to classical equality and inequality constraints or some particular geometric structure of $D$. Later, \cite{de2023constrained} further generalized the corresponding algorithmic framework and convergence theory to a broader class of nonconvex set-membership constrained optimization problems, of which \eqref{Eq:1} is a special case after an equivalent reformulation with a suitable auxiliary variable.

Second, \cite{kanzow2023inexact} employed an (inexact) penalty decomposition method for solving \eqref{Eq:1}. By introducing an auxiliary variable, carefully choosing the initial penalty value and updating the penalty parameter, the proposed algorithm demonstrated superior numerical performance than the safeguarded augmented Lagrangian method \cite{jia2023augmented}. Nevertheless, its practical performance depends on a careful choice of the penalty parameter, which needs to be tailored to the specific problem under consideration.

In these two classes of methods, under asymptotic Mordukhovich regularity (AM-regularity) being a sequential constraint qualification introduced in \cite{mehlitz2020asymptotic}, it is shown that every accumulation point satisfies the Mordukhovich stationarity conditions formulated via the limiting normal cone.

As a third class of methods, the constraint dissolving technique aims to design a constraint dissolving mapping that incorporates the constraint information into the objective function. This idea was originally proposed by \cite{xiao2024solving} for constructing penalty functions in (Stiefel) manifold optimization. Subsequently, constraint dissolving methods have been extended by Xiao et al. \cite{hu2024constraint,xiao2024dissolving} for solving optimization problems with equality constraints. More recently, \cite{xiao2026exact} further generalized this framework to equality-constrained optimization over convex sets. Despite these encouraging developments, existing constraint dissolving methods are still restricted to optimization problems with equality constraints and convex constraints. To the best of our knowledge, no constraint dissolving-type algorithm is currently available for the general optimization problem \eqref{Eq:1}. 

In this manuscript, we propose a constraint dissolving inexact penalty method for solving \eqref{Eq:1}, which can be viewed as an inexact penalty method enhanced by a constraint dissolving mapping, thereby benefiting from both the flexibility of inexact penalty methods and the geometric advantages offered by the constraint dissolving technique. The main contributions of this manuscript are summarized as follows. 
\begin{enumerate}
    \item We generalize the constraint dissolving technique, which was previously restricted to equality-constrained optimization \cite{hu2024constraint, xiao2024dissolving, xiao2026exact}, to optimization problems with general set-membership constraints $G(x)\in C$ and geometric constraints $x\in D$ , where $C$ is an arbitrary closed convex set and $D$ is allowed to be nonconvex. To the best of our knowledge, this is the first constraint dissolving framework capable of handling such a general class of optimization problems.
    \item We employ a mild constraint qualification, namely AM-regularity \cite{mehlitz2020asymptotic}, for both the construction of the constraint dissolving mapping and the convergence analysis of the corresponding algorithm. In contrast, AM-regularity is considerably weaker than the linear independence constraint qualification (LICQ) employed in \cite{xiao2024dissolving} and the assumption adopted in \cite[Assumption~1.1(2)]{xiao2026exact}, which can be viewed as an extension of the relaxed constant rank constraint qualification (rCRCQ), see \Cref{Lem:weakCQ} for a detailed comparison.
    \item We propose an \textit{inexact} penalty method combined with the constraint dissolving technique, where the penalty parameter is allowed to be updated adaptively. In contrast, the existing constraint dissolving methods \cite{hu2024constraint, xiao2024dissolving, xiao2026exact} are all based on exact penalty formulations with a fixed sufficiently large penalty parameter.
\end{enumerate}

This manuscript is organized as follows. We first introduce some preliminaries and properties of constraint dissolving mappings in \Cref{Sec:pre}. \Cref{Sec:alg} presents the constraint dissolving inexact penalty method and its convergence analysis. The implementation of constraint dissolving mappings for geometric constraints is discussed in \Cref{Sub:imp}. Numerical experiments are reported in \Cref{Sec:num}. We conclude the manuscript in \Cref{Sub:con}.

\section{Preliminaries}\label{Sec:pre}
In this manuscript, we adopt variational analysis as the primary tool. Let us first recall some definitions and properties, primarily following \cite{rockafellar1998variational}.

For any nonempty, closed, but not necessarily convex set $D \subset \mathbb R^n$, let us denote 
\begin{equation*}
    P_D(x):= \text{argmin}_{z\in D}\|z-x\| \quad \text{and}\quad \text{dist}_D(x):=\inf_{z\in D}\|z-x\|
\end{equation*}
as the corresponding projection mapping and the distance function, respectively. Since $D$ is closed and nonempty, $P_D(x)$ is nonempty for every $x$. However, when $D$ is nonconvex, $P_D(x)$ may fail to be a singleton, i.e., the projection onto $D$ does not need to be unique.

Let $C \in \mathbb R^n$ be nonempty, closed, and convex and let $G: \mathbb R^n \to \mathbb R^m$ be continuously differentiable. Then for each $x\in \mathbb R ^n$, by \cite[Lemma~2.1]{de2023constrained}, the gradient of $v(x):=\frac{1}{2}\text{dist}^2_C(G(x))$ can be represented as
\begin{equation}\label{Eq:gradientProjec}
    \nabla v(x)=\nabla G(x)^T\big(G(x)-P_C(G(x))\big).
\end{equation}
If $G(x)\in C$, then the corresponding gradient is trivially equal to $0$.

Given $z\in D$, the \emph{tangent cone} to $D$ at $z$ is given
\begin{equation*}
    \mathcal T_D(z):=\left\{ w \in \mathbb R^n \,\big|\, \frac{z^k-z}{\tau^k} \to w \quad \text{for some} \ z^k \xrightarrow[D]{} z, \tau^k \downarrow 0\right\},
    \end{equation*}
the \emph{limiting/regular tangent cone} to $D$ at $z$ is given by
\begin{equation*}
\mathcal T_D^{\mathrm{lim}}(z)
=
\left\{
\begin{array}{l|l}
w\in\mathbb R^n
&
\begin{array}{l}
\text{for all sequences } \bar z^k \xrightarrow[D]{} z
\text{ and } \tau^k\downarrow 0\\
\text{there exists}\ z^k \xrightarrow[D]{} z
\text{ such that }
\dfrac{z^k-\bar z^k}{\tau^k}\to w
\end{array}
\end{array}
\right\},
\end{equation*}
and consequently $\mathcal T_D^{\mathrm{lim}}(z) \subset \mathcal T_D(z)$.
The \emph{normal cone} to $D$ at $z$ is defined as
\begin{equation*}
    \mathcal N_D(z):=\{ v\in \mathbb R^n \,\big|\, \left<v,x-z\right>\leq o(\|x-z\|) \ \text{for} \ x\in D\},
\end{equation*}
the \emph{limiting normal cone} is defined as
\begin{equation*}
    \mathcal N_D^{\lim}(z):=\{ v\in \mathbb R^n \,|\, \text{there exist sequences} \ z^k \xrightarrow[D]{} z \ \text{and} \ v^k\to v \ \text{with} \ v^k\in \mathcal N_D(z^k)\}.
\end{equation*}
We set $\mathcal N_D(z)=\emptyset$ and $\mathcal N_D^{\lim}(z)=\emptyset$ if $z\notin D$. Note that the limiting normal cone is stable in the sense that
\begin{equation}\label{Eq:robustlimnormalcone}
    \limsup_{z\to \bar z} \mathcal N_D^{\lim}(z)=\mathcal N_D^{\lim}(\bar z) \quad \forall \bar z \in D.
\end{equation}
When $D$ is convex, then the limiting normal cone coincides with the corresponding standard normal cone, i.e., for $z\in D$,
\begin{equation*}
    \mathcal N_D^{\lim}(z)=\mathcal N_D(z):=\{ v\in \mathbb R^n \,|\, \left<v,x-z\right>\leq 0 \quad \forall{x\in D}\}.
\end{equation*}


\subsection{Stationarity and constraint qualification concepts}
This subsection recalls standard stationarity and regularity concepts for the prototypical optimization problem \eqref{Eq:1}. In the following, we denote
\[
\mathcal F:=\{x\in D \,|\, G(x)\in C\}
\]
as the feasible set of \eqref{Eq:1}. 

\begin{definition}\label{Def:MStat}
A point $ \bar x \in \mathcal F $ is called an \emph{M-stationary point} (Mordukhovich-stationary point) 
of \eqref{Eq:1} if there exists a multiplier $ \lambda \in \mathbb R^m $
such that 
\begin{equation*}
    0 \in \nabla f(\bar x) + \nabla G(\bar x)^T \lambda + \mathcal{N}_D^\textup{lim} (\bar x),
    \quad 
    \lambda \in \mathcal{N}_C ( G(\bar x) ).
\end{equation*}
\end{definition}

\begin{definition}\label{Def:AMStat}
A point $ \bar x \in \mathcal F $ is called an \emph{AM-stationary point} (asymptotically M-stationary point) 
of \eqref{Eq:1} if there exist sequences 
$\{ x^k \},\{\varepsilon^k\}\subset\mathbb R^n$ and $\{\lambda^k\},\{ z^k \} \subset\mathbb R^m$ such that
$x^k\to\bar x$, $\varepsilon^k\to 0$, $z^k\to 0$, as well as
\begin{equation*}
   \varepsilon^k \in \nabla f(x^k) + \nabla G(x^k)^T \lambda^k + \mathcal N_D^\textup{lim} (x^k),
   \quad\lambda^k \in \mathcal{N}_C( G(x^k) - z^k)
   \qquad
   \forall k\in\mathbb N.
\end{equation*}
\end{definition}

In order for an AM-stationary point to be M-stationary, let us recall the definition of the so-called asymptotic Mordukhovich regularity.

\begin{definition}\label{Def:AsymptoticRegularity}
	A point $\bar x\in\mathcal F $ of \eqref{Eq:1} is called \emph{AM-regular}
	(asymptotically Mordukhovich-regular) whenever the condition
	\[
		\limsup\limits_{x\to \bar x,\,z\to 0}\mathcal M(x,z)
		\subset
		\mathcal M(\bar x,0)
	\]
	holds, where $\mathcal M\colon\mathbb{R}^n\times\mathbb{R}^m\rightrightarrows\mathbb{R}^n$ 
	is the set-valued mapping defined via
	\[
		\mathcal M(x,z)
		:=
		\nabla G(x)^T\mathcal N_C(G(x)-z)+\mathcal N_D^\textup{lim}(x).
	\]
\end{definition}

The so-called AM-regularity serves as a constraint qualification, which is obviously satisfied if the linear independence constraint qualification (LICQ) holds for any $x\in \mathcal F$, i.e., $\nabla G(x)^T$ is full column rank in the case of $C := \{0\}$ and $D := \mathbb{R}^n$, which has been required in \cite{xiao2024dissolving}. As mentioned in \cite[Lemma~2.8]{jia2023augmented} and \cite[Theorem~3.10, 5.2]{mehlitz2020asymptotic}, the AM-regularity is easily verified whenever $C$ is a polyhedron, $G$ is an affine function, and $D$ is the union of finitely many polyhedrons.

According to \cite{xiao2024dissolving, xiao2026exact}, a constraint qualification seems to be necessary for the validity of constraint dissolving methods. In this manuscript, we merely require AM-regularity to hold. The central assumptions are summarized as follows.

\begin{assumption}(Properties of the optimization problem \eqref{Eq:1})
\begin{enumerate}[label=(\alph*)] \label{Ass:2}
    \item \label{item:2a} $f:\mathbb R^n \to \mathbb R$ and $G:\mathbb R^n\to \mathbb R^m$ are continuously differentiable. 
    \item \label{item:2b} $C \in \mathbb R^m$ is nonempty, closed, and convex. $D\in \mathbb R^n$ is nonempty, closed, but not necessarily convex.
    \item \label{item:2c} The AM-regularity, as defined in \Cref{Def:AsymptoticRegularity}, holds for any feasible point.
\end{enumerate}    
\end{assumption}

\subsection{Constraint dissolving mappings for nonconvex sets}
 In this part, we introduce constraint dissolving mappings for problems governed by nonconvex geometric constraints. To this end, in accordance to our prototypical problem \eqref{Eq:1}, we follow the notation as in \Cref{Ass:2}.


\begin{definition}[Constraint dissolving mapping]
    A mapping $\mathcal A :\mathbb R^n \to \mathbb R^n$ is called a \textit{constraint dissolving mapping} associated with the constraints of \eqref{Eq:1} if the following conditions hold.
    \begin{enumerate}[label=(\alph*)] \label{Ass:1}
    \item \label{item:a}$\mathcal A: \mathbb R^n \to \mathbb R^n$ is continuously differentiable. 
    \item \label{item:b} $\mathcal A(x)=x$ holds for all $x\in \mathcal F$.
    \item \label{item:c}$\nabla \mathcal A(x)^T \nabla G(x)^T \mathcal N_C (G(x))=\{0\}$ holds for all $x\in \mathcal F$.
\end{enumerate}
Moreover, there exists a 
    continuous mapping
$\mathcal R_{\mathcal A} : D\to \mathbb{R}^{n \times n}$
such that:
\begin{enumerate}[label=(\alph*),start=4]
    \item \label{item:d} \( \mathcal R_{\mathcal A}(x) = 0 \) for all \( x \in \mathcal F\),
   \item \label{item:e} for any \( x \in \mathcal F \), there exists \( \omega_x > 0 \) such that
    \[
    \left(\nabla \mathcal A(y) - I_n - \mathcal R_{\mathcal A}(y))\right)^T \mathcal N_D^{\text{lim}}(y) = \{0\}
    \]
   for all $ y \in D \text{ with } \|y - x\| \le \omega_x$.
   \end{enumerate}
\end{definition}

The existing constraint dissolving methods, for example \cite{xiao2024dissolving,xiao2026exact}, were applied to problems with equation constraint $G(x)=0$, i.e., $C:=\{0\}$. In this setting, the constraint dissolving mapping $\mathcal A$ is required to satisfy $\nabla \mathcal A(x)^T\nabla G(x)^T=0$, see \cite[Assumption~1.2]{xiao2024dissolving} and \cite[Assumption~1.2.2]{xiao2026exact}, which is a particular case of \ref{item:c}. Note that \ref{item:d} and \ref{item:e} are a generalization of \cite[Assumption~1.2.3]{xiao2026exact} to the nonconvex set $D$. Moreover, we only require the mapping $\mathcal R_{\mathcal A}$ to be continuous, rather than locally Lipschitz continuous as required in \cite[Assumption~1.2.3]{xiao2026exact}. 

In fact, ignoring the continuity requirement on the mapping $\mathcal R_{\mathcal A}$, our definition of the constraint dissolving mapping satisfying \Cref{Ass:1} reduces exactly to those given in \cite[Assumption~1.2]{xiao2024dissolving} and \cite[Assumption~1.2]{xiao2026exact} for the corresponding setting, thereby substantially extending them to set-membership and geometric constraints.

We next introduce some properties of the constraint dissolving mapping $\mathcal A$. Some of these results have already been established in \cite{xiao2026exact} under the restrictive constraint system and constraint dissolving mapping. For completeness, we provide detailed proofs adapted to our more general setting and to the generalized constraint dissolving mapping.

The following result is devoted to illustrate that $\nabla \mathcal A$ is tangentially identical.
\begin{proposition}\label{Prop:1}
If $\mathcal A$ is a constraint dissolving mapping, i.e., \Cref{Ass:1} holds,  then $\nabla \mathcal A (x) d=d$ for any $x\in \mathcal F$ and $d\in \mathcal T_{\mathcal F}(x)$.
\end{proposition}
\begin{proof}
For any $x\in \mathcal F$ and $d\in \mathcal T_{\mathcal F}(x)$, then there exist sequences $\{x^k\}\subset \mathcal F$ and $\{t_k\}$ satisfying $\lim_{k\to +\infty} x^k=x$ and $t_k \downarrow 0$ such that $ \lim_{k \to +\infty} {(x^k-x)}/{t_k}=d $.
Since $x\in \mathcal F$ and $\{x^k\}\subset \mathcal F$, then \Cref{Ass:1} \ref{item:b} implies $\mathcal A(x^k)-\mathcal A (x) = x^k-x$.

Dividing both sides by $t_k$ and passing the limit, using continuous differentiability of $\mathcal A$ gives
$\nabla \mathcal A(x)d = d$.
%
\end{proof}
The following result is directly deduced from \Cref{Prop:1}.
\begin{corollary}\label{Cor:1}
   Let \Cref {Ass:2} hold and let $\mathcal A$ be a constraint dissolving mapping. Then $\nabla \mathcal A (x) d=d$ for any $x\in \mathcal F$ and $d\in \left(\nabla G(x)^T\mathcal N_C(G(x))+\mathcal N_D^{\text{lim}}(x)\right)^{\perp}$.
\end{corollary}
\begin{proof}
Due to  and \Cref {Ass:2} we have
\begin{equation}\label{Eq:NormalconeFeasi}
    \mathcal N^{\text{lim}}_{\mathcal F}(x) \subset \nabla G(x)^T\mathcal N_C(G(x))+\mathcal N_D^{\text{lim}}(x),
\end{equation}
see \cite[Theorem~3.16]{mehlitz2020asymptotic}. Further, due to \Cref {Ass:2} \ref{item:2b}, the feasible set $\mathcal F$ is closed. Therefore, by \cite[Theorem~6.8~(b)]{rockafellar1998variational}, $\left(\mathcal N^{\text{lim}}_{\mathcal F}(x)\right)^{\circ}=\mathcal T^{\text{lim}}_{\mathcal F}(x)$ holds for all $x\in \mathcal F$. It follows from \eqref{Eq:NormalconeFeasi} that
    \begin{equation*}
    \begin{aligned}
        \left(\nabla G(x)^T\mathcal N_C(G(x))+\mathcal N_D^{\text{lim}}(x)\right)^{\perp} &\subset \left(\nabla G(x)^T\mathcal N_C(G(x))+\mathcal N_D^{\text{lim}}(x)\right)^{\circ} \\
        & \subset \left(\mathcal N^{\text{lim}}_{\mathcal F}(x)\right)^{\circ}=\mathcal T^{\text{lim}}_{\mathcal F}(x) \subset T_{\mathcal F}(x).
        \end{aligned}
    \end{equation*}
    Then the conclusion follows immediately from \Cref{Prop:1}.
\end{proof}

We establish the following properties of the constraint dissolving mapping $\mathcal A$.

\begin{proposition}\label{Prop:nablaA}
   Let \Cref {Ass:2} hold and let $\mathcal A$ be a constraint dissolving mapping. For any $x\in \mathcal F$, the following hold:
   \begin{enumerate}[label=(\arabic*)]
       \item \label{item:1prop} 
\begin{equation}\label{Eq:conesconstraintdis}
    \nabla \mathcal A (x)^T\left(\nabla G(x)^T \mathcal N_C (G(x))+\mathcal N^{\text{lim}}_D(x) \right)=\mathcal N^{\text{lim}}_D(x).
\end{equation}
Moreover the set $\nabla G(x)^T\mathcal N_C(G(x))+\mathcal N_D^{\text{lim}}(x)$ is invariant under the linear map $\nabla \mathcal A(x)^\top$.
\item \label{item:2prop} If $w\in \mathbb R^n$ is such that  $0\in \nabla \mathcal A(x)^Tw+ \mathcal N_D^{\text{lim}}(x)$, we have
\begin{align}\label{Item:squaretoD_2} 
           \nabla \mathcal A(x)^T w\in w +\nabla G(x)^T \mathcal N_C (G(x)).
       \end{align}
   \end{enumerate}
\end{proposition}
\begin{proof}
The proof of the first claim \ref{item:1prop} follows directly follows from \Cref{Ass:1} \ref{item:c}, \ref{item:d}, and \ref{item:e}. Due to $0\in \nabla G(x)^T\mathcal N_C(G(x))$, the invariance follows from $\mathcal N^{\text{lim}}_D(x) \subset \nabla G(x)^T\mathcal N_C(G(x))+\mathcal N_D^{\text{lim}}(x)$. This shows \ref{item:1prop}.

To show \ref{item:2prop}, we first prove an auxiliary result, namely
\begin{align}\label{eq:aux}
    \left(\nabla \mathcal A (x)^T\right)^m w =\nabla \mathcal A(x)^T w \quad \forall m = 1,2,\ldots.
\end{align}
    To see this, since $0\in \nabla \mathcal A(x)^Tw+ \mathcal N_D^{\text{lim}}(x)$, there exists $y\in \mathcal N_D^{\text{lim}}(x)$ such that $y=-\nabla \mathcal A(x)^Tw$. By \Cref{Ass:1} \ref{item:e}, we have
    \begin{equation}\label{Eq:squareinD}
    -\left(\nabla \mathcal A(x)^T\right)^2w=\nabla \mathcal A(x)^T y=y \in \mathcal N_D^{\text{lim}}(x).
    \end{equation}
    Therefore, $\left(\nabla \mathcal A(x)^T\right)^2w=\nabla \mathcal A(x)^Tw$. Multiplying both sides of \eqref{Eq:squareinD} by $\nabla \mathcal A(x)^T$ yields
    \begin{equation*}
    -\left(\nabla \mathcal A(x)^T\right)^3w=\nabla \mathcal A(x)^T y = y \in \mathcal N_D^{\text{lim}}(x).
    \end{equation*}
    Applying the same argument recursively, we obtain
\[
-\left(\nabla \mathcal A(x)^T\right)^m w = y
\quad \forall m = 1,2,\ldots,
\]
which implies
\[
\left(\nabla \mathcal A(x)^T\right)^m w
=
\nabla \mathcal A(x)^T w
\quad \forall m = 1,2,\ldots.
\]
This proves the auxiliary result \eqref{eq:aux}.

Let us now verify \ref{Item:squaretoD_2}. For any $d\in \mathcal T_{\mathcal F}(x)$, by \Cref{Prop:1}, one has   
\begin{equation*}
    \left< d, \nabla \mathcal A(x)^Tw-w \right>= \left< \left(\nabla \mathcal A(x)-I_n\right)d, w \right>=0,
\end{equation*}
it follows from \cite[Theorem~6.28~(a)]{rockafellar1998variational} and \eqref{Eq:NormalconeFeasi} that
\begin{equation*}
    \nabla \mathcal A(x)^Tw-w \in \left(\mathcal T_{\mathcal F}(x) \right)^{\circ} = \mathcal N_{\mathcal F}(x) \subset \mathcal N_{\mathcal F}^{\text{lim}}(x) \subset \nabla G(x)^T \mathcal N_C (G(x))+ \mathcal N_{D}^{\text{lim}}(x).
\end{equation*}
Consequently, there exist $\lambda \in \mathcal N_C (G(x))$ and $q\in \mathcal N_{D}^{\text{lim}}(x)$ such that
  $ \nabla \mathcal A(x)^Tw-w = \nabla G(x)^T \lambda +q $.
By \Cref{Ass:1} \ref{item:c} and \ref{item:e}, we have
\begin{equation*}
    0=\nabla \mathcal A(x)^T \left( \nabla \mathcal A(x)^T w-w\right)=\nabla \mathcal A(x)^T q =q.
\end{equation*}
Therefore, 
   $ \nabla \mathcal A(x)^Tw-w = \nabla G(x)^T \lambda \in \nabla G(x)^T \mathcal N_C (G(x))$.
This completes the proof.
\end{proof}
\begin{theorem}\label{Thm:2}
    Let \Cref {Ass:2} hold and let $\mathcal A$ be a constraint dissolving mapping.  Then for any $x\in \mathcal F$ and $w\in \mathbb R^n$, we have
    \[0\in \nabla \mathcal A(x)^Tw+\mathcal N_D^{\text{lim}}(x) \quad \Longleftrightarrow \quad 0\in w+\nabla G(x)^T\mathcal N_C(G(x))+\mathcal N_D^{\text{lim}}(x).
    \]
\end{theorem}
\begin{proof}
    ``$\Longleftarrow$" 
Multiplying both sides of the inclusion 
by $\nabla \mathcal{A}(x)^T$ and using \eqref{Eq:conesconstraintdis} of \Cref{Prop:nablaA} yields
    \begin{equation*}
    \begin{aligned}
        0&\in \nabla \mathcal A(x)^Tw+\nabla \mathcal A(x)^T\left(\nabla G(x)^T\mathcal N_C(G(x))+\mathcal N_D^{\text{lim}}(x)\right)\\
        &=\nabla \mathcal A(x)^Tw+\mathcal N_D^{\text{lim}}(x).
        \end{aligned}
    \end{equation*}
    
    ``$\Longrightarrow$" This direction directly follows from \eqref{Item:squaretoD_2}. To be specific, 
    \begin{equation*}
            0\in \nabla \mathcal A(x)^Tw+\mathcal N_D^{\text{lim}}(x)
            \subset w+ \nabla G(x)^T\mathcal N_C(G(x)) + \mathcal N_D^{\text{lim}}(x).
    \end{equation*}
\end{proof}

\subsection{Equivalence of stationarities}

We employ a constraint dissolving mapping $\mathcal A:\mathbb R^n \to \mathbb R^n$ to reformulate \eqref{Eq:1}, i.e.,
\begin{equation}\label{Eq:Q}\tag{Q}
\begin{aligned}
  \min \ &f(\mathcal A(x)) + \frac{\beta}{2}\text{dist}^2_C(G(x)) \\ 
  \text{s.t.} \ & x\in D,
\end{aligned}
\end{equation}
where $\beta>0$ denotes the penalty parameter.

This subsection is devoted to establishing the equivalence between M-stationary points of \eqref{Eq:1} and \eqref{Eq:Q}, and to analyzing the relationship between their AM-stationary points. To this end, we begin with recalling the main assumption of this part and by recalling the definitions of AM- and M-stationary points associated with \eqref{Eq:Q}.

\begin{assumption}\label{ass:stationarity}
    \Cref{Ass:2} holds and $\mathcal A$ be a constraint dissolving mapping in the sense of \Cref{Ass:1}.
\end{assumption}

\begin{definition}\label{Def:MQ}
    A point $ \bar x \in D $ is called an \emph{M-stationary point}
of \eqref{Eq:Q} if
\begin{equation*}
    0 \in \nabla \mathcal A(\bar x)^T \nabla f(\mathcal A(\bar x)) + \beta \nabla G(\bar x)^T\left(G(\bar x)-P_C(G(\bar x))\right)+\mathcal{N}_D^\textup{lim} (\bar x).
\end{equation*}
\end{definition}

\begin{definition}\label{Def:AMforQ}
A point $ \bar x \in D $ is called an \emph{AM-stationary point}
of \eqref{Eq:Q} if there exist sequences 
$\{ x^k \},\{\epsilon^k\}\subset\mathbb R^n$ such that
$x^k\to\bar x$, $\epsilon^k\to 0$, as well as
\begin{equation*}
   \epsilon^k \in \nabla \mathcal A(x^k)^T \nabla f(\mathcal A(x^k)) +  \beta \nabla G(x^k)^T \big(G(x^k)-P_C(G(x^k))\big)+\mathcal N_D^\textup{lim} (x^k)
   \qquad
   \forall k\in\mathbb N.
\end{equation*}
\end{definition}

\Cref{Def:AMforQ} implicitly requires that the sequence $\{x^k\}$ lies in $D$, since otherwise the corresponding limiting normal cone is empty. Consequently, the sequence $\{x^k\}$ is feasible for \eqref{Eq:Q}, but not necessarily feasible for \eqref{Eq:1}. The next result establishes the equivalence between the M-stationary points of \eqref{Eq:1} and \eqref{Eq:Q} at any 
feasible point of \eqref{Eq:1} (and hence of \eqref{Eq:Q}).


\begin{theorem}\label{Th:EquilM}
      Let \Cref{ass:stationarity} hold. For any $\bar x\in \mathcal F $, $\bar x$ is an M-stationary point of \eqref{Eq:1} if and only if $\bar x$ is an M-stationary point of \eqref{Eq:Q}.
\end{theorem}
\begin{proof}
    Since $C$ is convex, $\bar x\in \mathcal F$, and $\beta>0$, it follows that $\beta (G(\bar x)-P_C(G(\bar x))) \in \mathcal N_C(G(\bar x))$. Applying \Cref{Thm:2} with $w:=\nabla f(\mathcal A(\bar x))$, we obtain the equivalence of M-stationarity between \eqref{Eq:1} and \eqref{Eq:Q}.
\end{proof}
  Next, we establish the inclusion that every AM-stationary point of \eqref{Eq:Q} is AM-stationary for \eqref{Eq:1}.
  \begin{theorem}\label{Th:AM-statEquiv}
      Let \Cref{ass:stationarity} hold. For any $\bar x \in \mathcal F$, if $\bar x$ is an AM-stationary point of \eqref{Eq:Q}, then $\bar x$ is an AM-stationary point for \eqref{Eq:1}. 
  \end{theorem}
\begin{proof}
    Provided that $\bar x$ is AM-stationary for \eqref{Eq:Q}, by \Cref{Def:AMStat}, there exist sequences $\{x^k\}, \{\epsilon_1^k\} \subset \mathbb R^n$ satisfying $x^k\to \bar x$ and $\epsilon_1^k \to 0$ such that
\begin{equation}\label{Eq:AMforQ}
    \epsilon_1^k\in\nabla \mathcal A(x^k)^T \nabla f(\mathcal A(x^k))+\beta \nabla G(x^k)^T\big(G(x^k)-P_C(G(x^k))\big)+\mathcal N_D^{\text{lim}}(x^k) \quad \forall k\in \mathbb N.
\end{equation}
Taking the limit $k\to \infty$ to \eqref{Eq:AMforQ}, by the continuity of $\nabla \mathcal A, \nabla f, \nabla G,$ and $G$ as well as the outer semicontinuity of $\mathcal N_D^{\text{lim}}$, one has
    $0\in \nabla \mathcal A(\bar x)^T \nabla f(\bar x)+\mathcal N_D^{\text{lim}}(\bar x)$.
By \eqref{Item:squaretoD_2} with $w:=\nabla f(\bar x)$, there exists $\lambda \in \mathcal N_C(G(\bar x))$ such that
$\nabla \mathcal A(\bar x)^T \nabla f(\bar x) = \nabla f(\bar x) +\nabla G(\bar x)^T\lambda$.

Consequently, we have
\begin{small}
\begin{equation*}
    \begin{aligned}
       & \|\nabla \mathcal A (x^k)^T \nabla f(\mathcal A(x^k))-\nabla f(x^k)-\nabla G(x^k)^T\lambda \|\\
       & = \|\nabla \mathcal A (x^k)^T \nabla f(\mathcal A(x^k))-\nabla \mathcal A (\bar x)^T \nabla f(\bar x) - \nabla f(x^k) + \nabla f(\bar x)  - \nabla G(x^k)^T\lambda + \nabla G(\bar x)^T\lambda \|\\
       & \leq  \|\nabla \mathcal A (x^k)^T \nabla f(\mathcal A(x^k))-\nabla \mathcal A (\bar x)^T \nabla f(\bar x)\| + \|\nabla f(\bar x)- \nabla f(x^k) \| \\
       & ~~~~ +\| \nabla G(\bar x)^T\lambda - \nabla G(x^k)^T\lambda \|\to 0,
    \end{aligned}
\end{equation*}
\end{small}
where the limit holds due to the continuity of $\mathcal A$, $\nabla \mathcal A$, $\nabla f$, and $\nabla G$. Thus, by setting
\begin{equation*}
 \epsilon_2^k:=\nabla \mathcal A (x^k)^T \nabla f(\mathcal A(x^k))-\nabla f(x^k)-\nabla G(x^k)^T\lambda \quad \forall k\in \mathbb N
\end{equation*}
and
\begin{equation}\label{Eq:result2approx}
    \epsilon_3^k:=\beta \nabla G(x^k)^T\big(G(x^k)-P_C(G(x^k))\big) \quad \forall k\in \mathbb N,
\end{equation}
we have $\epsilon_2^k \to 0$ and $\epsilon_3^k \to 0$. Meanwhile,
\eqref{Eq:AMforQ} yields that
\begin{equation*}
    \epsilon_1^k-\epsilon_2^k-\epsilon_3^k \in \nabla f(x^k)+ \nabla G(x^k)^T\lambda + \mathcal N_D^{\text{lim}}(x^k) \quad \forall k\in \mathbb N. 
\end{equation*}
Let us define $\varepsilon^k:=\epsilon_1^k-\epsilon_2^k-\epsilon_3^k$ for all $k\in \mathbb N$, then $\varepsilon^k\to 0$. Since $\lambda \in \mathcal N_C(G(\bar x))$, define $z^k := G(x^k)- G(\bar x)$, then $z^k \to 0$ and $\lambda \in \mathcal N_C(G(x^k)-z^k)$ for all $k\in \mathbb N$. Therefore, by \Cref{Def:AMStat}, $\bar x$ is an AM-stationary point of \eqref{Eq:1}.
\end{proof}
  
\begin{corollary}\label{Th:PtoQ}
      Let \Cref{ass:stationarity} hold. For any $\bar x \in \mathcal F$, if $\bar x$ is an AM-stationary point of \eqref{Eq:Q}, then $\bar x$ is an M-stationary point for \eqref{Eq:1}.
\end{corollary}
\begin{proof}
   By arguments similar to those used in \Cref{Th:AM-statEquiv}, one has
\begin{equation*}
    0\in \nabla \mathcal A(\bar x)^T \nabla f(\bar x)+\mathcal N_D^{\text{lim}}(\bar x).
\end{equation*}
Since $G(\bar x)\in C$, then $\bar x$ is an M-stationary point of \eqref{Eq:Q}. Therefore, by \Cref{Th:EquilM}, $\bar x$ is also an M-stationary point of \eqref{Eq:1}.
\end{proof}

In fact, every M-stationary point of \eqref{Eq:Q} is AM-stationary of \eqref{Eq:Q}. Combined with the above analysis and under mild conditions, this implies that, for any fixed $\beta>0$, as typically considered in exact penalty methods, a point $x\in \mathcal F$ is M-stationary for \eqref{Eq:Q} if and only if it is AM-stationary for \eqref{Eq:Q}.

Let us note that all the results in this section rely on AM-regularity defined in \Cref{Def:AsymptoticRegularity}. In contrast, stronger constraint qualifications were required in previous works. Specifically, for the setting $C:=\{0\}$ and $D:=\mathbb R^n$, considered in \cite{xiao2024dissolving}, the linear independence constraint qualification (LICQ) was employed. By \cite[Lemma~2.7]{jia2023augmented}, AM-regularity is weaker than the relaxed constant positive linear dependence constraint qualification (RCPLD), and therefore also weaker than LICQ. Moreover, for the setting $C:=\{0\}$ and $D$
being a convex set, considered in \cite{xiao2026exact}, the following constraint qualification, stated using the notation of this manuscript, was assumed:
\begin{definition}\label{Def:rCRCQ}
    There exists $r>0$ such that at every point $x\in \mathcal F$, there exists $\tau_x>0$ satisfying the following conditions:
    \begin{itemize}
        \item[(a)] For any $y\in \{y\in \text{\rm aff}(D) \,|\, \|y-x\|\leq \tau_x\}$, it holds that $\text{\rm dim}(\{\nabla G(y)d\,|\,d\in \mathcal E\})=r$ with $\mathcal E:=\text{\rm aff}(D)-x$.
        \item[(b)] For any $y\in \{y\in D \,|\, \|y-x\| \leq \tau_x\}$, it holds that
        \begin{equation*}
            \text{\rm dim}\left(\{\nabla G(y)d \,|\, d\in \text{\rm lin}(\mathcal T_D(y))\}\right) =r
        \end{equation*}
        with $\text{\rm lin}(\mathcal T_D(y)):=\mathcal T_D(y) \cap - \mathcal T_D(y)$.
    \end{itemize}
\end{definition}
Since $D$ is convex and $\mathcal E$ is a subspace, by an easy computation, we have $\mathcal T_D(y)\subset \mathcal E$ and $\text{lin}(\mathcal T_D(y))=\mathcal N_D(y)^{\perp}$ for any $y\in D$. Consequently, we have \begin{equation}\label{Eq:resultE}
\mathcal E^{\perp} \subset \mathcal N_D(y) \quad \forall y\in D,
\end{equation}
and thus
\begin{equation}\label{Eq:resultscone}
  \text{lin}(\mathcal T_D(y)) =  \mathcal N_D(y)^{\perp} \subset (\mathcal E ^{\perp})^{\perp} =\mathcal E. 
\end{equation}
On the other hand, by \Cref{Def:rCRCQ}, for all $y\in D$ satisfying $\|y-x\|\leq \tau_x$, one has
\begin{equation*}
\text{\rm dim}
\left(\{\nabla G(y)d\,|\,d\in \mathcal E\}\right)= \text{\rm dim}\left(\{\nabla G(y)d \,|\, d\in \text{\rm lin}(\mathcal T_D(y))\}\right) =r.
\end{equation*}
It, together with \eqref{Eq:resultscone}, implies that
\begin{equation}\label{Eq:result_rCRCQ}
    \nabla G(y)\mathcal E = \nabla G(y)\text{\rm lin}(\mathcal T_D(y)) \quad \forall y\in D \ \text{satisfying}\ \|y-x\|\leq \tau_x.
\end{equation}

In the following, we show that, in the particular case where $C:=\{0\}$ and $D$ is convex, AM-regularity is weaker than \Cref{Def:rCRCQ}, employed in \cite{xiao2026exact}. Specifically, for this setting, AM-regularity defined in \Cref{Def:AsymptoticRegularity} reduces to
$ 
    \limsup_{x \to \bar x} \mathcal K(x) \subset \mathcal K(\bar x)
$ 
with 
$   \mathcal K (x) : = \text{Range} \left(\nabla G(x)^T\right)+\mathcal N_D(x)$,
 for any feasible point $x$, i.e, any point $x\in D$ satisfying $G(x)=0$.

\begin{lemma}\label{Lem:weakCQ}
    Let $\bar x \in D$ be such that $G(\bar x)=0$, i.e., $\bar x$ is feasible for the setting $C:=\{0\}$ and assume further that $D$ is convex. If \Cref{Def:rCRCQ} holds at $\bar x$, then $\bar x$ is AM-regular.
\end{lemma}
\begin{proof}
    For any $\xi\in \limsup_{x \to \bar x} \mathcal K(x)$, there exist sequences $\{\xi^k\}$ and $\{x^k\}$ satisfying $\xi^k \to \xi$ and $x^k \xrightarrow[D]{} \bar x$ such that $\xi^k \in \mathcal K(x^k)$ for all $k\in \mathbb N$. Particularly, there exist sequences $\{\lambda^k\}$ and $\{\eta^k\}$ satisfying $\lambda^k\in \mathbb R^n$ and $\eta^k\in \mathcal N_D(x^k)$ such that
\begin{equation*}
    \xi^k = \nabla G(x^k)^T \lambda ^k +\eta^k \quad \forall k\in \mathbb N.
\end{equation*}
Since $\mathcal E$ is a subspace and $\nabla G(x^k)\mathcal E$ is also, there exist $\lambda_1^k \in \nabla G(x^k)\mathcal E $ and $\lambda_2^k \in \left( \nabla G(x^k)\mathcal E \right)^{\perp}$ such that $\lambda^k=\lambda_1^k +\lambda_2^k \ \forall k\in \mathbb N$,
and consequently
\begin{equation*}
    \xi^k = \nabla G(x^k)^T \lambda_1 ^k + \nabla G(x^k)^T \lambda_2 ^k+ \eta^k \quad \forall k\in \mathbb N.
\end{equation*}
Since $\lambda_2^k \perp \nabla G(x^k)\mathcal E$, it follows from \eqref{Eq:resultE} that $\nabla G(x^k)^T \lambda_2^k \in \mathcal E^{\perp} \subset \mathcal N_D (x^k)$. Let us now set $u^k:=\nabla G(x^k)^T \lambda_2 ^k+ \eta^k$ for all $k\in \mathbb N$, since $D$ is convex, then $u^k\in \mathcal N_D(x^k)$. Meanwhile, we have
\begin{equation}\label{Eq:formxi}
    \xi^k = \nabla G(x^k)^T \lambda_1 ^k + u^k \quad \forall k\in \mathbb N.
\end{equation}

Suppose that the sequence $\{(\lambda_1^k, u^k)\}$ is unbounded. Dividing 
 \eqref{Eq:formxi} by $\|(\lambda_1^k, u^k)\|$ and taking the limit $k\to \infty$ yields that there exist a non-vanishing multiplier $(\bar \lambda_1, \bar u)$ which satisfies $\bar \lambda_1\in \mathbb R^n$ and $\bar u \in \mathcal N_D(\bar x)$ such that 
 \begin{equation}\label{Eq:0limit}
     0 =\nabla G (\bar x)^T \bar \lambda_1 + \bar u.
 \end{equation}
 Since $\bar u \in \mathcal N_D(\bar x)$, then for any $d\in \text{lin}(\mathcal T_D(\bar x))$, we have $\left<d, \bar u \right>=0$. It follows from \eqref{Eq:0limit} that
   $  0=\left<\nabla G (\bar x)^T \bar \lambda_1, d\right> = \left<\bar \lambda_1, \nabla G(\bar x)d \right>$,
 thus \eqref{Eq:formxi} implies $
     \bar \lambda_1 \perp \nabla G(\bar x) \text{lin}(\mathcal T_D(\bar x)) = \nabla G(\bar x) \mathcal E$.
 Let us recall that $\lambda_1^k \in \nabla G(x^k)\mathcal E$ and $\mathcal E$ is a subspace, then we have $\bar \lambda_1 \in \nabla G(\bar x) \mathcal E$. Therefore, $\bar \lambda_1 =0 $ and consequently $\bar u=0$ by \eqref{Eq:0limit}, which yields a contradiction with the fact that $(\bar \lambda_1, \bar u)$ is non-vanishing.

 Thus, the sequence $\{(\lambda_1^k, u^k)\}$ is bounded and therefore, possesses a convergent subsequence with limit $(\hat \lambda_1, \hat u)$ which satisfies $\hat \lambda_1 \in \mathbb R^n$ and $\hat u \in \mathcal N_D (\bar x)$. By taking the limit in \eqref{Eq:formxi} while respecting $\xi^k \to \xi$, the continuity of $\nabla G$ and $\mathcal N_D(\cdot)$, we have
 \begin{equation*}
     \xi = \nabla G(\bar x)^T\hat \lambda_1 +\hat u \in \text{Range} (\nabla G(\bar x)^T)+ \mathcal N_D(\bar x)=\mathcal K(\bar x).
 \end{equation*}
 This shows that $\bar x$ is AM-regular.
\end{proof}
\section{Algorithm and convergence}\label{Sec:alg}
This section focuses on a projection-based constraint dissolving method for solving \eqref{Eq:1}, along with the corresponding convergence analysis. 
For convenience, we define $h_{\beta}:\mathbb R^n \to \mathbb R$ as
\begin{equation}\label{Eq:h}
h_{\beta}(x):=f(\mathcal A(x))+\frac{\beta}{2}\text{dist}_C^2(G(x)).
 \end{equation}
The associated gradient is given by
\begin{equation*}
  \nabla_x h_{\beta}(x)=\nabla \mathcal A(x)^T \nabla f(\mathcal A(x)) +\beta \nabla G(x)^T \left(G(x)-P_C(G(x))\right).
\end{equation*}
Under \Cref{Ass:1} \ref{item:a} and \Cref{Ass:2} \ref{item:2a}, \ref{item:2b}, $h_{\beta}(x)$ is continuously differentiable in terms of $x$. Note that $h_{\beta}$ does not include the geometric constraint $x\in D$, which will be imposed explicitly as a constraint. The overall method is outlined in \Cref{Alg:IP}.

\begin{algorithm}[H]\caption{Constraint dissolving inexact penalty method}
	\label{Alg:IP}
    \small
	\begin{algorithmic}[1]
		\REQUIRE $\beta_0>0$ and $x_0\in D$.
            \STATE Set $k:=0$.
		\WHILE{$x^k$ does not satisfy a termination criterion}
        \STATE Compute an approximate M-stationary point $x^{k+1}$ of the subproblem
        \begin{equation}\label{Eq:subph}
\begin{aligned}
  \min_{x} \ & h_{\beta_k}(x) \\ 
  \text{s.t.} \ & x\in D,
\end{aligned}
\end{equation}
i.e., for some suitable (sufficiently small) vector $\varepsilon^{k+1}\in \mathbb R^n$, $x^{k+1}$ needs to satisfy
            \begin{equation}\label{Eq:subp}
      \varepsilon^{k+1}\in \nabla_x h_{\beta_k}(x^{k+1})+\mathcal N_D^{\text{lim}}(x^{k+1}).
            \end{equation}
            \STATE Choose $\beta_{k+1}\geq\beta_k$. Set $k=k+1$.
        \ENDWHILE
		\RETURN $x^k$
	\end{algorithmic}
\end{algorithm}
Note that the main computational burden of \Cref{Alg:IP} lies in Step 3, which can be realized by a projected gradient method, where the computation of the projection onto $D$ is necessary.

\Cref{Alg:IP}, as an inexact penalty method, always suffers from the drawback that accumulation points may be not feasible for the original problem. But, the following results illustrate some useful information in the sense that such an accumulation point is a stationary point for the constraint violation problem of \eqref{Eq:1}.
\begin{proposition}\label{Prop:feas}
    Let \Cref{Ass:2} \ref{item:2a} and \Cref{Ass:1}\ref{item:a} hold, and $\{x^k\}$ be a sequence generated by \Cref{Alg:IP} with 
$\{\varepsilon^k\}$ being bounded and $\beta_k\to \infty$. Then every accumulation point $\bar{x}$ of the sequence 
$\{x^k\}$ is an M-stationary point of the feasibility problem
\begin{equation}\label{Eq:feasliftp}
\min_{x} \;
\frac{1}{2}\operatorname{dist}_C^2\big(G(x)\big)
\quad \text{\rm s.t.}\ x\in D.
\end{equation}
\end{proposition}

\begin{proof}
    Let $\{x^{k+1}\}_{\mathcal I}$ be a subsequence converging to $\bar x$. \eqref{Eq:subp} is equal to
    \begin{equation}\label{Eq:samesubp}
    \begin{aligned}
  \varepsilon^{k+1}& \in \nabla \mathcal A(x^{k+1})^T \nabla f(\mathcal A(x^{k+1})) +\beta_k \nabla G(x^{k+1})^T \left(G(x^{k+1})-P_C(G(x^{k+1}))\right)\\
  &~~~~+ \mathcal N_D^{\text{lim}}(x^{k+1})
       \end{aligned}
    \end{equation}
    for all $k\in \mathbb N$.
    Since $\mathcal N_D^{\text{lim}}(x^{k+1})$ is a cone, dividing the above inclusion by $\beta_k$ yields that
    \begin{equation*}
        \begin{aligned}
      & \frac{\varepsilon^{k+1}}{\beta_k}\in \frac{1}{\beta_k}
      \nabla \mathcal A(x^{k+1})^T \nabla f(\mathcal A(x^{k+1}))\\
       &~~~~+\nabla G(x^{k+1})^T(G(x^{k+1})-P_C(G(x^{k+1}))+\mathcal N_D^{\text{lim}}(x^{k+1})
       \end{aligned}
    \end{equation*}
    for all $k\in \mathbb N$. Meanwhile, by \Cref{Ass:2} \ref{item:2a} and \Cref{Ass:1}\ref{item:a}, $\nabla f$, $\mathcal A$, $\nabla \mathcal A$, $G$, $\nabla G$, and $P_C$ are continuous. Consequently, we have $\nabla \mathcal A(x^{k+1})^T \nabla f(\mathcal A(x^{k+1})) \to_{\mathcal I} \nabla \mathcal A(\bar x)^T \nabla f(\mathcal A(\bar x))$. Therefore, taking the limit $k \to_{\mathcal I} \infty$ and using the robustness of the limiting normal cone \eqref{Eq:robustlimnormalcone}, the fact $\beta_k \to \infty$ yields that
    \begin{equation*}
        0\in\nabla G(\bar x)^T \left( G(\bar x)-P_C(G(\bar x)) \right)+\mathcal N_D^{\text{lim}}(\bar x).
    \end{equation*}
    It shows that $\bar x$ is an M-stationary point of \eqref{Eq:feasliftp}.
\end{proof}

\Cref{Prop:feas} does not establish the feasibility of accumulation points of the sequence generated by \Cref{Alg:IP}, it only guarantees their M-stationarity. Under the current assumptions, this appears to be the strongest result that can be established. As noted in \cite{kanzow2023inexact}, if the feasible set $\mathcal F$ of \eqref{Eq:1} is nonempty and the function $\text{dist}_C^2(G(x))$ is convex, then every M-stationary point is a global minimizer of the feasibility problem and, consequently, is feasible for \eqref{Eq:1}. Therefore, additional assumptions on $G$ and/or $C$ are required to guarantee feasibility. For instance, $G$ is affine.

The following result illustrates some extra conditions on the function $h_{\beta}(x)$ which can guarantee that the accumulation point is feasible.
\begin{proposition}
    Let \Cref{Ass:2} \ref{item:2a} and \Cref{Ass:1}\ref{item:a} hold, $\{x^k\}$ be a sequence generated by \Cref{Alg:IP} with $\beta_k\to \infty$ and $\bar x$ be an accumulation point of this sequence. If there exists some $B\in \mathbb R$ satisfying $h_{\beta_k}(x^{k+1}) \leq B$ for all $k\in \mathbb N$, then $\bar x$ is feasible for \eqref{Eq:1}.
\end{proposition}
\begin{proof}
    Since $\bar x$ is the accumulation point of $\{x^k\}$, there is a subsequence, for example, $\{x^{k+1}\}_{\mathcal I}$ for some index set $\mathcal I \subset \mathbb N$, satisfying $x^{k+1} \to_{\mathcal I} \bar x$. By \eqref{Eq:subph} and the closedness of $D$, we obtain $\bar x\in D$. It remains to verify that $G(\bar x)\in C$.

    Let us recall back \eqref{Eq:h}, that is
    \begin{equation}\label{Eq:1forfeas}
        \text{dist}_C^2 (G(x^{k+1}))=\frac{2\left(h_{\beta_k}(x^{k+1})-f(\mathcal A(x^{k+1}))\right)}{\beta_k} \leq \frac{2\left(B-f(\mathcal A(x^{k+1}))\right)}{\beta_k}
    \end{equation}
    for all $k\in \mathbb N$. In view of \Cref{Ass:1}\ref{item:a} and \Cref{Ass:2} \ref{item:2a}, both $f$ and $\mathcal A$ are continuous, which imply that $f(\mathcal A (x^{k+1}))\to_{\mathcal I} f(\mathcal A(\bar x))$. Let us take the limit as $k\to_{\mathcal I}\infty$ in \eqref{Eq:1forfeas}, the assumption $\beta_k \to \infty$ implies that $\text{dist}_C(G(\bar x))=0$ where we used the continuity of $G$ according to \Cref{Ass:2} \ref{item:2a}. Since $C$ is a closed set, it follows that $G(\bar x)\in C$. Therefore, $\bar x$ is feasible for \eqref{Eq:1}.
\end{proof}

As mentioned above, $h_{\beta}(x)$ is continuously differentiable. If we further consider $D$ is a compact set, then the subproblem \eqref{Eq:subph} always has the global minimizer. So we now consider $x^{k+1}$ as the (approximate) global minimizer of \eqref{Eq:subph}, for any $x\in \mathcal F$, one has
\begin{equation*}
    h_{\beta_k}(x^{k+1}) \leq h_{\beta_k}(x)=f(\mathcal A(x))+\frac{\beta_k}{2} \text{dist}^2_C(G(x))=f(\mathcal A(x))=f(x) \leq B
\end{equation*}
for some suitable $B$. 

The main convergence result is established as follows.
\begin{theorem}\label{Th:con}
    Let \Cref{Ass:2} and \Cref{ass:stationarity} hold. Let $\{x^k\}$ be a sequence generated by \Cref{Alg:IP} with 
$\{\varepsilon^k\}\to 0$ and $\beta_k \to \infty$. Suppose that $\bar x$ be a feasible accumulation point of this sequence, then $\bar{x}$ is an M-stationary point of the optimization problem \eqref{Eq:1}.
\end{theorem}
\begin{proof}
    Let $\{x^{k+1}\}_{\mathcal I}$ be a subsequence converging to $\bar x$. \eqref{Eq:samesubp} yields that
    \begin{equation*}
    \begin{aligned}
  \varepsilon^{k+1}& \in \nabla \mathcal A(x^{k+1})^T \nabla f(\mathcal A(x^{k+1})) + \nabla G(x^{k+1})^T \mathcal N_C\left(P_C(G(x^{k+1}))\right)
  + \mathcal N_D^{\text{lim}}(x^{k+1})
       \end{aligned}
    \end{equation*}
    for all $k\in \mathbb N$. Let us set $z^k:=G(x^k)-P_C(G(x^k))$ for any $k\in \mathbb N$, then $z^{k+1} \to_{\mathcal I} 0$. Therefore, by \Cref{Ass:2} \ref{item:2c}, one has
    \begin{equation}\label{Eq:optimalityQ}
        0\in \nabla \mathcal A(\bar x)^T \nabla f(\bar x)+\nabla G(\bar x)^T \mathcal N_C(G(\bar x))+\mathcal N_D^{\text{lim}}(\bar x).
    \end{equation}
    Let us claim that $\nabla \mathcal A(\bar x)^T \nabla f(\bar x) \in \nabla f(\bar x) + \nabla G(\bar x)^T \mathcal N_C(G(\bar x))$. Since $\mathcal N_C(G(\bar x))$ and thus $\nabla G(\bar x)^T \mathcal N_C(G(\bar x))$ are convex cones, we have
      $0\in \nabla f(\bar x) + \nabla G(\bar x)^T \mathcal N_C(G(\bar x))+ \mathcal N_D^{\text{lim}}(\bar x)$.
Therefore, $\bar x$ is an M-stationary point of \eqref{Eq:1}. It remains us to verify $\nabla \mathcal A(\bar x)^T \nabla f(\bar x) \in \nabla f(\bar x) + \nabla G(\bar x)^T \mathcal N_C(G(\bar x))$. 

In view of \eqref{Eq:optimalityQ}, there exist $\xi_1\in \nabla G(\bar x)^T\mathcal N_C(G(\bar x))$ and $\eta_1 \in \mathcal N_D^{\text{lim}}(\bar x)$ such that
    $\nabla \mathcal A(\bar x)^T \nabla f(\bar x)+\xi_1+\eta_1=0$.
By \Cref{Ass:1} \ref{item:c}, we have $
    \left( \nabla \mathcal A(\bar x)^T\right)^2 \nabla f(\bar x)+\eta_1=0$.
Thus,
\begin{equation}\label{Eq:1forTh}
    \left( \nabla \mathcal A(\bar x)^T\right)^2 \nabla f(\bar x) -\nabla \mathcal A(\bar x)^T \nabla f(\bar x) = -\eta_1+\xi_1+\eta_1 = \xi_1.
\end{equation}
On the other hand, for any $d\in \mathcal T_{\mathcal F}(\bar x)$, \Cref{Prop:1} implies that
\begin{equation*}
    \left< d,  \nabla \mathcal A (\bar x)^T \nabla f(\bar x)-\nabla f(\bar x) \right> = \left<\left(\nabla \mathcal A(\bar x)-I_n \right)d, \nabla f(\bar x) \right>=0,
\end{equation*}
thus, by \eqref{Eq:NormalconeFeasi}, we have
\begin{equation*}
\begin{aligned}
    \nabla \mathcal A (\bar x)^T \nabla f(\bar x)-\nabla f(\bar x) &\in (\mathcal T_{\mathcal F}(\bar x))^{\perp} \subset (\mathcal T_{\mathcal F}(\bar x))^{\circ} =\mathcal N_{\mathcal F}(\bar x)\\
    & \subset \mathcal N^{\text{lim}}_{\mathcal F}(\bar x) \subset \nabla G(\bar x)^T \mathcal N_C (G(\bar x))+\mathcal N_D^{\text{lim}}(\bar x).
    \end{aligned}
\end{equation*}
Then there exist $\xi_2\in \nabla G(\bar x)^T \mathcal N_C (G(\bar x))$ and $\eta_2 \in \mathcal N_D^{\text{lim}}(\bar x)$ such that
\begin{equation}\label{Eq:2forTh}
    \nabla \mathcal A (\bar x)^T \nabla f(\bar x)-\nabla f(\bar x) = \xi_2+\eta_2.
\end{equation}
By \Cref{Ass:1} \ref{item:c}, we have
   $\nabla \mathcal A (\bar x)^T \left( \nabla \mathcal A (\bar x)^T \nabla f(\bar x)-\nabla f(\bar x) \right) = \eta_2$.
Combining it with \eqref{Eq:1forTh} gives $\eta_2=\xi_1$. Consequently, \eqref{Eq:2forTh}, together with the fact that $\nabla G(\bar x)^T \mathcal N_C(G(\bar x))$ is a convex cone, implies
\begin{equation*}
 \nabla \mathcal A (\bar x)^T \nabla f(\bar x)-\nabla f(\bar x)=\xi_2+\xi_1 \in \nabla G(\bar x)^T \mathcal N_C (G(\bar x)).
\end{equation*}
This completes the proof.
 \end{proof}
 
 \Cref{Alg:IP} can be regarded as an inexact penalty method, where the sequence $\{\beta_k\}\to \infty$ is used to penalize the constraint $G(x^k)\in C$. If an exact penalty method is considered instead, a fixed penalty parameter $\beta>\beta^*$ is introduced to replace the possibly increasing parameter $\beta_k$ at each iteration, where $\beta^*$ is a constant dependent on the properties of $f$ and $\mathcal A$. In this case, the accumulation point $\bar x$ given in \Cref{Th:con} can be shown by straightforward computation to be an AM-stationary point of \eqref{Eq:Q}, and by \Cref{Th:PtoQ}, consequently AM- and M- stationary point of \eqref{Eq:1}.

Let us note that the objective function of subproblems \eqref{Eq:subph} is continuously differentiable and the set $D$ is nonempty and closed. We therefore employ the general spectral gradient method proposed in \cite[Algorithm~3.1]{jia2023augmented} to solve subproblems \eqref{Eq:subph} arising in \Cref{Alg:IP}. The well-definedness of the general spectral gradient method has been established in \cite[Proposition~3.1]{jia2023augmented}. Moreover, according to \cite[Theorem~3.4]{jia2023augmented}, if the sublevel set is bounded and the inner loop always terminates, then the sequence generated by the general spectral gradient method converges (along a subsequence) to an M-stationary point of the subproblem \eqref{Eq:subph}. It provides a theoretical guarantee for computing the approximately M-stationary point of the subproblem \eqref{Eq:subph}.

\section{Implementation of constraint dissolving mappings for nonconvex sets}\label{Sub:imp}
As discussed above, the original idea of the constraint mapping $\mathcal A(x)$ stems from (Stiefel) manifold optimization \cite{xiao2024solving}, where $C:=\{0\}$. Extending this structure to a general constraint set $C$ (not necessarily equal to $\{0\}$), the construction of $\mathcal A(x)$ naturally depends on the projection $P_C(G(x))$. However, the nondifferentiability of $P_C(G(x))$ poses the significant technical challenges that remain unsolved currently. In the following of this manuscript, we therefore focus on the case $C:=\{0\}$. This setting still captures a broad class of relevant practical problems,for example, optimal control \cite{harder2021reformulation}, inverse problems \cite{basir2022physics}, and related areas \cite{higham2002computing,xiao2024dissolving,xiao2026exact,qi2006quadratically}. 

To construct a constraint dissolving mapping $\mathcal A(x)$ tailored to the constraints $G(x)=0$ and $x\in D$, we rely on the following notion. 
\begin{definition}\label{Ass:Q}
We say that $\mathcal Q: \mathbb R^n \to \mathbb R^{n \times n}$ is a projective mapping if
\begin{enumerate}[label=(\alph*)]
    \item \label{item:Qa}$\mathcal Q(x)$ is continuously differentiable,
    \item \label{item:Qb}$\mathcal Q(x)$ is symmetric positive semidefinite for any $x\in \mathbb R^n$,
    \item \label{item:Qc}$\text{Null} \left(\mathcal Q(x)\right)= \text{Span} \left(\mathcal N^{\text{lim}}_D(x) \right) $ for all $x\in D$.
\end{enumerate}
\end{definition}

Observe that the notion of a projective mapping solely relies on the set $D$, which reduces to \cite[Assumption~4.3]{xiao2026exact} when $D$ is convex. 
We summarize in \Cref{Tab:Q} the construction of projective mappings $\mathcal Q$ tailored to several widely used settings of $D$. For details on the computation, we refer to \Cref{App:1}.

\begin{table}[t]
\centering
\renewcommand{\arraystretch}{1.3}
\begin{tabular}{p{2.2cm}|l|l}
\hline
Constraint set & Formulation & Possible choices of $\mathcal Q$ \\
\hline
Disjunctive &
\(
\begin{aligned}
\big\{ &(w,v) \in \mathbb R^n \,|\, (w_i, v_i) \in T \\
&\forall i=\{1, \ldots, {n}/{2}\}\big\}
\end{aligned}
\)
&
\(
\begin{aligned}
&\mathcal Q(w,v)=\\
&c\operatorname{diag}(w_1^2,v_1^2,\ldots,w_{n/2}^2,v_{n/2}^2)
\end{aligned}
\)
\\sparsity &
$\left\{x\in \mathbb R^n \,\big|\, \|x\|_0 \leq \kappa \right\}$ &
$\mathcal Q(x)=c\operatorname{diag}(x_1^2,\ldots,x_n^2)$
\\
Low-rank &
$\left\{ X \in \mathbb R^{p \times q} \,\big|\, \text{rank} (X) \leq \kappa \right\}$ &
$\mathcal Q(X)[Y]= U_1U_1^TY+YV_1V_1^T$
\\
PSD low-rank &
$\{X\in \mathbb S^{p} \,\big|\, \text{rank} (X) \leq \kappa \}$ &
$\mathcal Q(X)[Y]=U_1U_1^TY+YU_1U_1^T$
\\
\hline
\end{tabular}%
\caption{Realizations of the projective mapping $\mathcal Q$ for several geometric constraint sets. Here
    $T:=\left\{(s,t) \,\big|\, s\in [\delta_1, \delta_2], \ t\in [\tau_1, \tau_2], \ st=0 \right\}$
 with constants $\delta_1, \delta_2$, $\tau_1$, and $\tau_2$ satisfying $-\infty \leq \delta_1, \tau_1 \leq 0$ and $0 < \delta_2, \tau_2 \leq \infty$. $c>0$ is an arbitrary constant. $\kappa$ denotes the corresponding target sparsity level  or the target rank. Let $X = U\Lambda V^T$ be the SVD of $X$, with $U=(U_1, U_2)$, $V=(V_1, V_2)$ split at the leading $\kappa$ columns and $\Lambda \in \mathbb R^{p\times q}$ diagonal. For symmetric $X$, $V=U$.}
\label{Tab:Q}
\end{table}

Based on the projective mapping $\mathcal Q(x)$, a constraint dissolving mapping $\mathcal A(x)$ for constraints $G(x)=0$ and $x\in D$ can be formulated as
\begin{equation}\label{Eq:A}
    \mathcal A(x):=x-\mathcal Q(x)\nabla G(x)^T \left( \nabla G(x)\mathcal Q(x)\nabla G(x)^T+\alpha \| G(x)\|^2 I_m\right)^{-1} G(x),
\end{equation}
where $\alpha>0$ is a constant, which is introduced to guarantee the invertibility of $ \nabla G(x)\mathcal Q(x)\nabla G(x)^T$ $+\alpha \| G(x)\|^2 I_m$, as $\mathcal Q(x)$ is positive semidefinite by \Cref{Ass:Q} \ref{item:Qb}. Let us note that \eqref{Eq:A} degenerates to \cite[(4.2)]{xiao2026exact} when $D$ is convex. Moreover, to guarantee the differentiability of $\mathcal A(x)$, we require that $G:\mathbb R^n \to \mathbb R^m$ is twice continuously differentiable. Moreover, it is necessary to verify that $\nabla G(x)\mathcal Q(x)\nabla G(x)^T$ is positive definite for all $x\in \mathcal F$ with $C:=\{0\}$ such that $\mathcal A(x)$ is well-defined. To this end, we introduce the following constraint qualification condition,
\begin{equation}\label{Eq:strictCQ}
     \text{Range}\left(\nabla G(x)^T\right) \cap \text{Span} \left(\mathcal N_D^{\text{lim}}(x)\right)=\{0\},
\end{equation}
where $x\in \mathcal F$ with $C:=\{0\}$. This condition implies NNAMCQ (no nonzero abnormal multiplier constraint qualification) by a straightforward calculation and also \Cref{Ass:2} \ref{item:2c}, see \cite[Section~5.1]{mehlitz2020asymptotic}. Furthermore, it is still weaker than \cite[Assumption~1.1 (2b)]{xiao2026exact} by \cite[Lemma~2.2]{xiao2026exact}, which was used there to ensure the positive definiteness of  $\nabla G(x)\mathcal Q(x)\nabla G(x)^T$ for all $x\in \mathcal F$ with $C:=\{0\}$. The following result shows that the same conclusion remains valid under the weaker condition \eqref{Eq:strictCQ}.
\begin{lemma}\label{Lem:Q}
     Let the constraint qualification \eqref{Eq:strictCQ} hold at any $x\in \mathcal F$ with $C:=\{0\}$. Then for any mapping $\mathcal Q:\mathbb R^n\to \mathbb R^{n \times n}$ satisfying \Cref{Ass:Q}, one has $\nabla G(x)\mathcal Q(x)\nabla G(x)^T \succ 0$ for all $x\in \mathcal F$ with $C:=\{0\}$.
\end{lemma}
\begin{proof}
    By contradiction, we assume that there exists $x\in \mathcal F$ with $C:=\{0\}$ such that $\nabla G(x)\mathcal Q(x)\nabla G(x)^T$ is not full rank. Without loss of generality, there exists $d\in \mathbb R^m$ satisfying
        $0\neq \nabla G(x)^T d$ such that $ \mathcal Q(x)\nabla G(x)^Td=0$.
    By \Cref{Ass:Q} \ref{item:Qc}, we have
    \begin{equation*}
        0\neq \nabla G(x)^T d \in \text{Null} \left(\mathcal Q(x)\right)= \text{Span} \left( \mathcal N_D^{\text{lim}}(x) \right).
    \end{equation*}
    Since $\nabla G(x)^T d \in \text{Range} (\nabla G(x)^T)$, then \eqref{Eq:strictCQ} implies that $\nabla G(x)^T d =0$, yielding a contradiction. Therefore, $\nabla G(x)\mathcal Q(x)\nabla G(x)^T \succ 0$ for all $x\in \mathcal F$ with $C:=\{0\}$.
\end{proof}
We now verify that the mapping $\mathcal A$ defined in \eqref{Eq:A} satisfies \Cref{Ass:1}. 
\begin{proposition}\label{Eq:validA}
    Let \Cref{Ass:2} \ref{item:2b} hold $C:=\{0\}$ and \eqref{Eq:strictCQ} be satisfied. Then for any projective mapping $\mathcal Q:\mathbb R^n\to \mathbb R^{n \times n}$ in the sense of \Cref{Ass:Q}, the corresponding mapping $\mathcal A: \mathbb R^n \to \mathbb R^n$ defined via \eqref{Eq:A} is a constraint dissolving mapping in the sense of \Cref{Ass:1}.
\end{proposition}
\begin{proof}
    Let us first show that the mapping $\mathcal A$ satisfies \Cref{Ass:1}. By \Cref{Ass:Q} \ref{item:Qa} and \Cref{Ass:2} \ref{item:2b}, $\mathcal A$ is continuously differentiable. Therefore, \Cref{Ass:1} \ref{item:a} holds.

    For any $x\in \mathcal F$ with $C:=\{0\}$, then we have $x\in D$ and $G(x)=0$, which imply that $\mathcal A(x)=x$.

    In order to verify \Cref{Ass:1} \ref{item:c}, for convenience, let us define 
    \begin{equation}\label{Eq:Jx}
        J(x):= \mathcal Q(x)\nabla G(x)^T \left( \nabla G(x)\mathcal Q(x)\nabla G(x)^T+\alpha \| G(x)\|^2 I_m\right)^{-1}.
    \end{equation}
   Then, for any $x\in \mathcal F$ with $C:=\{0\}$, we have 
   \begin{equation}\label{Eq:Jx2}
   J(x) = \mathcal Q(x)\nabla G(x)^T \left( \nabla G(x)\mathcal Q(x)\nabla G(x)^T\right)^{-1},
   \end{equation}
   which is valid in view of \Cref{Lem:Q} and consequently 
    \begin{equation}
    \begin{aligned}
    \nabla \mathcal A(x)&=I_n- \mathcal D J(x)[G(x)]-J(x)\nabla G(x)
    =I_n-J(x)\nabla G(x)  
    \end{aligned}
\end{equation}
holds for $x\in \mathcal F \ \text{with}\ C:=\{0\}$. Here, $\mathcal D J(x)$ denotes a third-order tensor, we did not explicitly formulate it since $\mathcal D J(x)[G(x)]=0$ holds for all $x$ satisfying $G(x)=0$. In view of \Cref{Ass:Q} \ref{item:Qb} and \ref{item:Qc}, we have $\mathcal Q(x)^T \mathcal N_D^{\text{lim}}(x)=\{0\}$.
Therefore, for any fixed $x\in \mathcal F$ with $C:=\{0\}$, choose arbitrarily $\eta\in \nabla \mathcal A(x)^T \nabla G(x)^T \mathcal N_C\left( G(x)\right)$, then there exists $\lambda \in \mathcal N_C\left( G(x)\right)$ such that
$ \eta =  \nabla \mathcal A(x)^T \nabla G(x)^T \lambda$. Thus, we have
\begin{equation*}
    \begin{aligned}
        \eta&=\left(I_n-\nabla G(x)^TJ(x)^T \right)\nabla G(x)^T \lambda\\
        &=\nabla G(x)^T \lambda
        -\nabla G(x)^T \left(\nabla G(x)\mathcal Q(x)^T \nabla G(x)^T \right)^{-1}\nabla G(x) \mathcal Q(x)^T\nabla G(x)^T\lambda\\
        & = \nabla G(x)^T \lambda -\nabla G(x)^T \lambda =0.
    \end{aligned}
\end{equation*}
It follows that $\nabla \mathcal A(x)^T \nabla G(x)^T \mathcal N_C\left( G(x)\right) \subset \{0\}$.
Since $0\in \mathcal N_C\left( G(x)\right)$, the inclusion
$\{0\} \subset \nabla \mathcal A(x)^T \nabla G(x)^T \mathcal N_C\left( G(x)\right)$
holds trivially. Therefore, we conclude that $\nabla \mathcal A(x)^T \nabla G(x)^T \mathcal N_C\left( G(x)\right)=\{0\}$.

Therefore, $\mathcal A(x)$ defined in \eqref{Eq:A} satisfies \Cref{Ass:1} \ref{item:a}-\ref{item:c}, it remains to verify that $\mathcal A(x)$ satisfies \Cref{Ass:1} \ref{item:d} and \ref{item:e}.

By \eqref{Eq:A} and \eqref{Eq:Jx}, we know that
\begin{equation*}
    \nabla \mathcal A(y)=I_n-\mathcal D J(y)[G(y)]-J(y)\nabla G(y) \quad \forall y\in \mathbb R^n.
\end{equation*}
Let us now define $\mathcal R_{\mathcal A}(y):=-\mathcal D J(y)[G(y)]$, which is continuous since $G(y)$ is twice continuously differentiable and $J(y)$ is continuously differentiable. For all $x \in \mathcal F$ with $C:=\{0\}$, which means $G(x)=0$, consequently we have $\mathcal R_{\mathcal A}(x):=-\mathcal D J(x)[G(x)]=0$. Thus, \Cref{Ass:1} \ref{item:d} is satisfied.

In this case, we have $ \nabla \mathcal A(y)-I_n-\mathcal R_{\mathcal A}(y)=-J(y)\nabla G(y) \quad \forall y\in \mathbb R^n$.
Let us recall $\mathcal Q(y)^T \mathcal N_D^{\text{lim}}(y)=\{0\}$ for all $y\in D$, and consequently $J(y)^T\mathcal N_D^{\text{lim}}(y)=\{0\}$ by \eqref{Eq:Jx2}. Thus 
we have
\begin{equation}\label{Eq:7}
    \begin{aligned}
        \left( \nabla \mathcal A(y)-I_n-\mathcal R_{\mathcal A}(y)\right)^T\mathcal N_D^{\text{lim}} (y)=-\nabla G(y)^T J(y)^T\mathcal N_D^{\text{lim}} (y)=\{0\} \quad \forall y\in D.
    \end{aligned}
\end{equation}
Overall, the mapping $\mathcal A$ defined in \eqref{Eq:A} satisfies \ref{item:e}. This completes the proof.
\end{proof}

\section{Numerical experiments}\label{Sec:num}
This section presents numerical results for several classes of examples introduced in \Cref{Tab:Q}, aiming to demonstrate the effectiveness of \Cref{Alg:IP} (the constraint dissolving inexact penalty method, abbreviated as CDiP), in comparison with the safeguarded augmented Lagrangian method (ALM) proposed in \cite{jia2023augmented} and the penalty decomposition algorithm (PD) studied in \cite{kanzow2023inexact}. All experiments are implemented in \textsc{MATLAB R2023B} on a MacBook Air (Apple M4, 24 GB RAM) running macOS Sonoma. 

For a fair comparison, we make every effort to align parameter settings across of the three algorithms. To be specific, all methods share the same initial point $x^0\in D$ and the same initial penalty value which is chosen as
\begin{equation*}
    \beta_0:=P_{[10^{-3}, 10^3]} \left( 10 \frac{\max \{ 1, f(x^0)\}}{\max\left\{1, \frac{1}{2}\|G(x^0)\|^2\right\}}\right),
\end{equation*}
following the strategy in \cite[p.~153]{birgin2000nonmonotone} and \cite{jia2023augmented}. The penalty parameter is then updated via $\beta_{k+1}=\eta \beta_k$
for both CDiP and PD. Since the choice of the growth rate $\eta$ is crucial for the overall performance of penalty-based methods, we customize it for each class of problems and report its value on each test problem. For ALM, the update rule of penalty parameter follows exactly that in \cite{jia2023augmented}.
Note that $\beta$ corresponds to $\rho$ in ALM from \cite{jia2023augmented} and $\tau$ in PD from \cite{kanzow2023inexact}. For notational convenience, we use $\beta$ throughout our numerical experiments. We impose an upper bound to the penalty value of $10^8$ for all three algorithms.

 In terms of termination criteria,
 CDiP and ALM adopt the same stopping condition
$\|G(x^{k+1})\| \leq \varepsilon_{\text{out}}$,
while PD uses $\|x^{k+1}-y^{k+1}\|+\|G(x^{k+1})\| \leq \varepsilon_{\text{out}}$.
    Both criteria are designed to measure the infeasibility of the iterates. Unless specified, we set $\varepsilon_{\text{out}}:=10^{-4}$ for toy examples and $10^{-5}$ for real-world problems. The exact values will be reported for each testproblem. Besides the termination criteria, we imposed a maximum of $10^7$ cumulative function evaluations. The algorithm is terminated once this limit was reached.

A key component of CDiP is the constraint dissolving technique. At each iteration $k$, the corresponding mapping is defined as in \eqref{Eq:A}:
\begin{equation*}
    \mathcal A(x^k):=x-\mathcal Q(x^k)\nabla G(x^k)^T \left( \nabla G(x^k)\mathcal Q(x^k)\nabla G(x^k)^T+\alpha \| G(x^k)\|^2 I_m\right)^{-1} G(x^k).
\end{equation*}
Here, the projective mapping $\mathcal Q(x^k)$ is chosen according to the class of the text problems, following the rules discussed in \Cref{Tab:Q}.

Regarding the subproblem solvers, both CDiP and ALM employ the spectral gradient method \cite[Algorithm~3.1]{jia2023augmented}, with parameters and stopping criterion consistent with those in \cite[Section~6]{jia2023augmented}. For PD, to ensure fairness, we used a gradient descent method with Armijo line search for the $x$-update step (instead of BFGS or L-BFGS as in \cite{kanzow2023inexact}). The subproblem solver terminates once  the tolerance $\varepsilon_{\text{inner}_k}$ is reached, unless specified, where $\varepsilon_{\text{inner}_k}:={10^{-4}}/{\sqrt{k+1}}$. 

\subsection{Complementarity-constrained examples}
\subsubsection{Academic Rosenbrock problem}
    \rm Let us first consider a simple (two-dimensional) Rosenbrock problem with an MPCC constraint,
    \begin{equation*}
    \begin{aligned}
        \min_x \ &(1-x(1))^2+100*\left(x(2)-x(1)^2\right)^2\\
        \text{s.t.} \ & e^Tx=1,\ x(1)\geq 0, \ x(2) \geq 0, \ x(1)x(2)=0.
        \end{aligned}
    \end{equation*}
    It admits the global optimal solution $x^*=(0,0)^T$. We chose $\mathcal Q$ as defined in \Cref{Tab:Q}. Based on this example, we aim to test the sensitivity of CDiP with respect to the parameter $c$ in $\mathcal Q$. To this end, we set $c\in \{0.01, 0.1, 1, 10, 100\}$ for comparison. For each setting, we generated 1000 (random) starting points, which are located near the Rosenbrock valley to avoid highly ill-conditioned directions. Specifically, we construct them as follows $
        w=\text{linspace}(-1,1, 1000)$, $v= w.^2+0.1 \ \text{rand}(1000)$,
    and set $x^0= [w_i,v_i]^T, i=1, \ldots, 1000$. The starting points were obtained by projecting $x^0$ onto $D$. We set $\eta=1.1$ and $\varepsilon_{\text{out}}=10^{-4}$, the results are listed in \Cref{Tab:resultsforRosenbrock}.  Avg.$k$ and Avg.$i$ denote the average numbers of outer and inner iterations, respectively.  ``\# at $(0,0)^T$" and ``\# at others" represent the numbers of runs converging to $(0,0)^T$ and to other points, respectively.
\begin{table}[htbp]
\centering
\renewcommand{\arraystretch}{1.2}
\begin{tabular}{|c|c|c|c|c|}
\hline
c & Avg.$k$ & Avg.$i$ &  \# at $(0,0)^T$ & \# at others \\ \hline
0.01 & 14.3 & 18.8 & 931 & 69 \\ 
0.1 & 12.4 & 13.6 & 934 & 66 \\ 
1 & 11.7 & 11.9 & 937 & 63 \\ 
10 & 11.1 & 11.2 & 940 & 60 \\
100 & 10.8 & 10.9& 942 & 58 \\
\hline
\end{tabular}
\caption{Convergence of CDiP for different values of $c$ in $\mathcal Q$. We report average numerical results for the Rosenbrock problem with an MPCC constraint over 1000 random starting points. }
\label[table]{Tab:resultsforRosenbrock}
\end{table}

\Cref{Tab:resultsforRosenbrock} indicates that CDiP is relatively insensitive to the choice of $c$ for the MPCC-constrained Rosenbrock problem. Over a wide range of values of $c$, we find that larger values of $c$ tend to reduce both the average numbers of outer and inner iterations. Moreover, the number of runs converging to $(0,0)^T$ increases slightly as the parameter $c$ increases.

\subsubsection{Obstacle problem}
    \rm We consider a discretized obstacle problem studied in \cite[Section~7.4]{harder2021reformulation} and later revised in \cite[Example~6.2]{jia2023augmented}. Let $x:=(w,v,z)$ and the problem is formulated as
    \begin{equation*}
    \begin{aligned}
        \min_x \ &f(x):=\frac{1}{2}\|w\|^2-e^Tv+\frac{1}{2}\|v\|^2\\
        \text{s.t.} \ & w \geq 0,\ -Av-w+z=0,\ v \geq 0, \ z\geq 0, \ v^Tz=0,
        \end{aligned}
    \end{equation*}
    where $A$ is a tridiagonal matrix arising from a discretization of the negative operator in one dimension, i.e., $a_{ii}=2$ for all $i$ and $a_{ij}=-1$ for all $i=j \pm 1$. $e$ denotes the all-one vector of approximate size. It can be readily verified that any feasible point satisfies the constraint qualification \eqref{Eq:strictCQ} and \Cref{Ass:2} \ref{item:2c}.

    We treated the constraint $w\geq 0$ as part of the set constraint $x\in D$ here. According to \cite[Proposition~6.41]{rockafellar1998variational}, we chose the projective mapping $\mathcal Q$ as
    \begin{equation}
        \mathcal Q(x) =
        \begin{pmatrix}
            \text{diag}({w.}^2) & 0\\
            0 & \mathcal Q(v,z),
        \end{pmatrix}
    \end{equation}
    where is $\mathcal Q(v,z)$ is defined in \eqref{Eq:Qdisjunctive} with $c=1$. 
    We set $A\in \mathbb R^{m \times m}$ with $m\in \{10, 20, 30\}$ and used the projection of all-one ($3m$-dimensional) vector onto $D$ as the starting points, the parameter growth rate $\eta=2$. We chose $\varepsilon_{\text{out}}=10^{-4}$ as the tolerance of outer loop, the results are reported in \Cref{Tabl:Obstacle}. The number of outer iterations is denoted by $k$, and $i_{\text{cum}}$ denotes the accumulated number of inner iterations. The quantity $\text{Avg}.i:=i_{\text{cum}} / k$ represents the average number of inner iterations per outer iteration. $f_\text{val}$ and $\beta$ denote the function value and the penalty value at the final iterate, respectively. The runtime is reported as ``time(s)” in seconds.

\begin{table}[htbp]
\centering
\renewcommand{\arraystretch}{1.2}
\begin{tabular}{l@{\hskip 10pt}rrrrrrr}
\hline
P.($m$) & Alg. & $k$ & $i_{\text{cum}}$ & Avg.$i$ & $f_{\text{val}}$ & $\beta$ & time($s$) \\
\hline
$m=10$ & CDiP & 14 & 19,028 & 1,359 & -3.34e-3 & 819,200 & 10.00 \\
 & ALM & 9 & 42,405 & 4,711 & -2.68e-3 & 500,000 & 10.83 \\
 & PD  & 13 & 421,161 & 32,397 & -1.96e-2 & 409,600 & 9.59 \\[4pt]
$m=20$ & CDiP & 16 & 197,449 & 12,340 & -1.04e-2 & 6553,600 & 136 \\
& ALM & 15 & 209,485 & 13,695 & -0.39e-2 & 1000,000 & 38.59\\
& PD & 14 & 4619,382 & 329,955 & -1.23e-1 & 1638,400 & 107 \\[4pt]
$m=30$  & CDiP & 16 & 673,103 & 42,068 & -4.85e-2 & 9830,400 & 659 \\
& ALM & 12 & 468,469 & 39,039& -8.88e-2 & 1500,000 & 38.71 \\
& PD & -- & -- & --& --& --& --\\
\hline
\end{tabular}
\caption{Numerical results for the discretized obstacle problem.} \label{Tabl:Obstacle}
\end{table}

When $m=10$ and $m=20$, CDiP consistently required fewer inner iterations and attained lower objective value compared to both ALM and PD. Since CDiP and ALM employed the same subproblem solver with identical parameter settings, this improvement can be attributed, at least in part, to the constraint dissolving mechanism (together with the approximate penalty parameter). When $m=30$, however, ALM performed best, while PD fails to terminate within 1200 s and its results are omitted (marked as ``--" in \Cref{Tabl:Obstacle}). The performance of CPiD deteriorates when the penalty parameter becomes large, since the objective term $f(\mathcal A(x^k))$ becomes dominated by the penalty term, reducing the benefit of constraint dissolving. Moreover, PD generally requires more inner iterations and hence smaller step sizes, suggesting that constraint dissolving can mitigate the adverse effect of large penalty parameters.

CDiP incurs additional overhead due to the repeated evaluation of the constraint dissolving mapping and its derivative at each iteration, where the latter may involve a third-order tensor or a Jacobian--vector product. Although this cost is partly offset by the reduction in inner iterations for $m=10,20$, it is still desirable to develop derivative-free methods with convergence guarantees as the subproblem solver for CDiP, in order to improve the scalability of CDiP to high-dimensional (and ill-conditioned) problems.

\subsection{Sparsity-constrained examples}
In this section, we discuss the application of our method to problems with sparsity constraints.
\subsubsection{Sparse portfolio problem}
\rm We consider the following sparse portfolio optimization problem:
\begin{equation}\label{Eq:sparsityportfolio}
\min_x \ \frac{1}{2} x^TQx - c^Tx \quad \text{s.t.}\ e^Tx=1, \ x \geq 0, \ \|x\|_0 \leq s,
\end{equation}
where $e\in \mathbb R^n$ denotes the vector of all ones, $Q$ and $c$ denote the covariance matrix and the expected returns of $n$ possible assets, and $s>0$ measures the sparsity level. Note that every feasible point satisfies the constraint qualification \eqref{Eq:strictCQ} and \Cref{Ass:2} \ref{item:2c}.

The data $Q$ and $c$ are taken from the test problem collection \cite{frangioni2007sdp}. We considered instances with $n\in \{200, 300\}$ and chose $s=5$ for $n=200$ and $s=7$ for $n=300$, with a total of 60 test problems. Since \eqref{Eq:sparsityportfolio} can be reformulated as a mixed-integer quadratic program, we use Gurobi \cite{gurobi} to obtain the benchmark solutions. The time limit of Gurobi is set to 30 minutes for each test instance. We compare CDiP with ALM and PD on all test problems. To ensure a fair comparison, all three methods employ the same warm-start strategy. Specifically, we first solve \eqref{Eq:sparsityportfolio} without the sparsity set $\{x \,|\, \|x\|_0 \leq s \}$, which reduces to a convex quadratic program and can be solved easily. The obtained solution is then projected onto the sparsity set and the resulting point is used as the initial point for the corresponding test problem. 

The optimal function values obtained by CDiP, ALM, and PD are compared with the corresponding values found by Gurobi within the prescribed time limit. For all methods, we set $\varepsilon_{\mathrm{out}}=10^{-5}$ as the termination tolerance. For CDiP and PD, the penalty parameter is chosen as $\eta=2$. Moreover, for CDiP, we set $c=0.1$ to construct$\mathcal Q$ in \Cref{Tab:Q}.
However, PD failed to terminate within 30 minutes for all test problems. Therefore, its numerical results are omitted from the subsequent comparison. \Cref{fig:portfolio} reports the objective value gaps of CDiP and ALM relative to the value returned by Gurobi.
\begin{figure}[htb]
\centering
\makebox[\textwidth][c]{%
\includegraphics[width=1.25\textwidth]{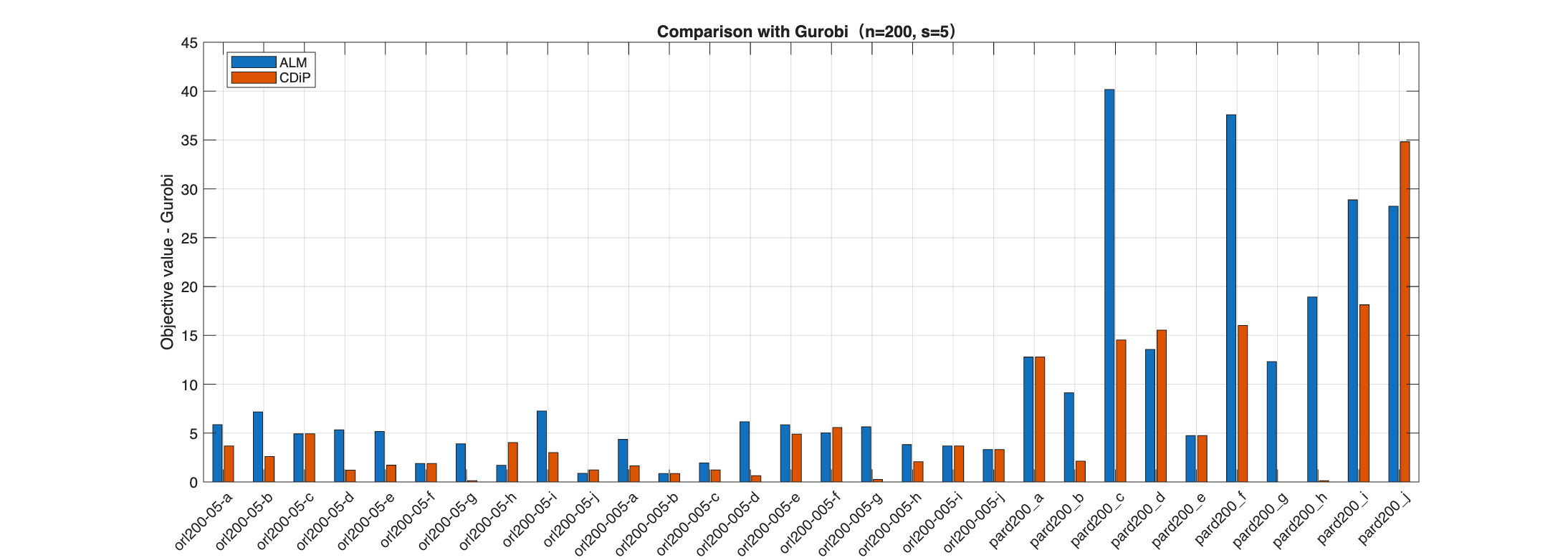}
}
\vspace{2mm}
\makebox[\textwidth][c]{%
\includegraphics[width=1.25\textwidth]{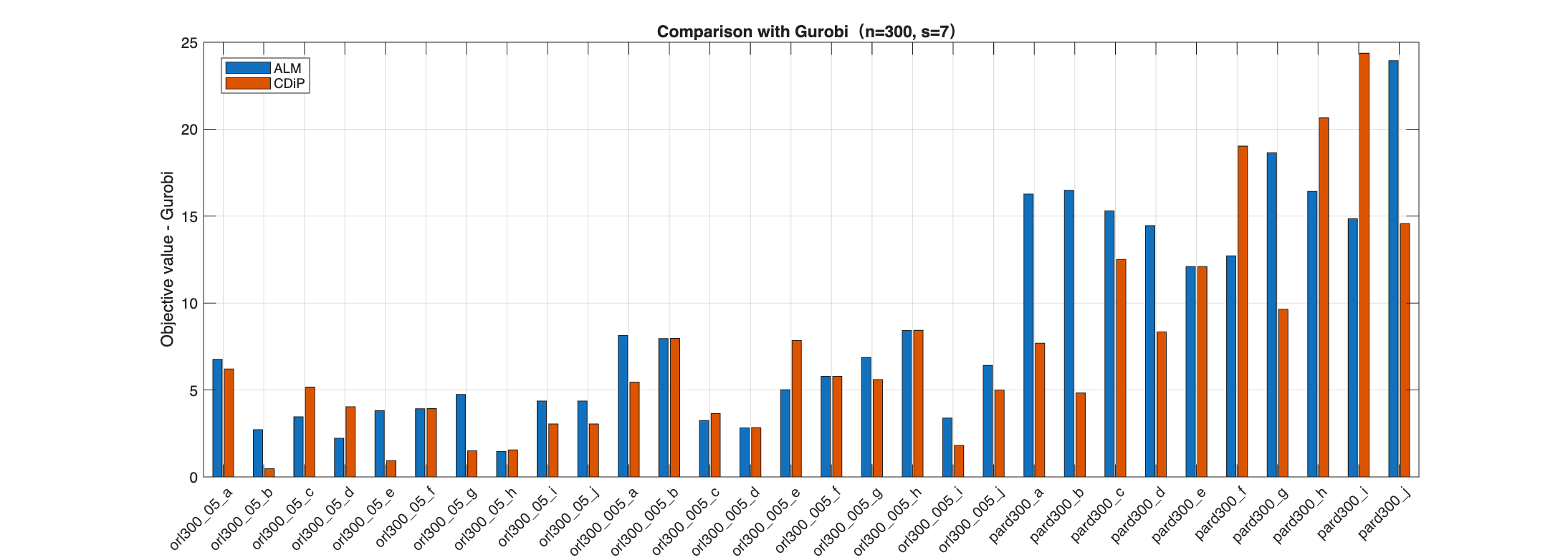}
}
\caption{Optimal function value gaps of CDiP and ALM with respect to Gurobi for test problems with $n=200$ (top) and $n=300$ (bottom).}
    \label{fig:portfolio}
\end{figure}

The results in \Cref{fig:portfolio} suggest that CDiP tended to produce lower objective values than ALM. Specifically, CDiP achieved better objective values than ALM on 18 out of 30 instances for dimension $n=200$ and on 16 out of 30 instances for dimension $n=300$. On the remaining instances, ALM produced lower objective values, but the differences from those obtained by CDiP were comparatively small. Moreover, there are 5 instances where the optimal value gap between CDiP and Gurobi is below 0.5, four of which correspond to the case $n=200$. In contrast, the smallest objective value gap achieved by ALM was approximately 0.8. Overall, these observations indicate a slight advantage of CDiP over ALM in terms of solution quality for the considered portfolio optimization problems. 
\subsubsection{Greenhouse temperature control problem}
\rm We consider a temperature control problem in a greenhouse. The goal is to maintain the greenhouse temperature close to a prescribed reference temperature under changing environmental conditions. Heating and cooling are used as the control inputs, and simultaneous heating and cooling are prohibited. 
%
Let us set $x:=(T, u_{\text{heat}}, u_{\text{cool}})\in \mathbb R^{3n}$, the corresponding decritized optimal control problem is formulated as 
\begin{equation}\label{Eq:discrGreenhous}
\begin{aligned}
    \min_{x}
\quad & \frac{1}{2}\|T-T_{\text{ref}}\|^2
\\
\text{s.t.}\quad
& T=c_1Au_{\text{heat}}-c_2 A u_{\text{cool}}+\lambda A T_{\text{env}}+ T_0b, \\
& 0\leq u_{\text{heat}}, u_{\text{cool}} \leq e, \\
& (u_{\text{heat}}, u_{\text{cool}})\in D:=\Pi_{i=1}^n D_i ,
\end{aligned}
\end{equation}
where $T, T_{\text{ref}} \in \mathbb R^n$ denote the greenhouse temperature and the desired reference temperature, respectively, and $e\in \mathbb R^n$ denotes the vector of all ones. The scalar $T_0\in \mathbb R$ denotes the inertial greenhouse temperature. $u_{\text{heat}}, u_{\text{cool}}$ represent the heating and cooling control inputs, respectively. The parameter $\lambda>0$ characterizes the heat exchange between the greenhouse and the environment, $c_1>0, c_2>0$ are the efficiency coefficients of the heating and cooling actuators, respectively. $A\in \mathbb R^{n \times n}$ and $b\in \mathbb R^n$ arise from the descretization. Here, $A$ is a strictly lower triangular matrix with entries $a_{ij}=\Delta t(1-\Delta t\lambda)^{i-j-1}$ if $i>j$ and $a_{ij}=0$ if $i \leq j$ and $b$ has components $b_i=(1-\Delta t \lambda)^{i-1}$, where $\Delta t>0$ is the corresponding descretization parameter.

To enforce that heating and cooling cannot be activated simultaneously, we define
\begin{equation*}
    D_i:=\{ ({u_{\text{heat},i}}, {u_{\text{cool},i}}) \,|\, \|({u_{\text{heat},i}}, {u_{\text{cool},i}}) \|_0 \leq 1\}.
\end{equation*}

We used the ambient temperature data in Chemnitz, Germany, from 10 May 2019 12:00:00 to 10 June 2019 12:00:00. The reference temperature in the greenhouse was set as $T_{\text{ref}}=20^\circ\mathrm{C}$. The time horizon was uniformly discretized into $n=30$ intervals, yielding a time step of $\Delta t=25.6552$h. We compared the performance of CDiP, ALM, and PD. For all three algorithms, a warm-start strategy was employed. Specifically, we first ignored the sparsity constraint, thereby obtaining a convex quadratic programming problem. The resulting solution was then projected onto $D$ and used as the starting point. The penalty update parameter was set to $\eta =10$ for both CDiP and PD, while $c=0.1$ was used to construct $\mathcal Q$ in CDiP. It is worth noting that PD, as a decomposition method, failed to produce complementary heating and cooling inputs. More specifically, simultaneous heating and cooling inputs occurred 29 times. Therefore, only the optimal inputs of CDiP and ALM are reported in
\Cref{fig:twofigsGreenhouse}.

\begin{figure}[htbp]
    \centering
    \includegraphics[width=0.48\textwidth]{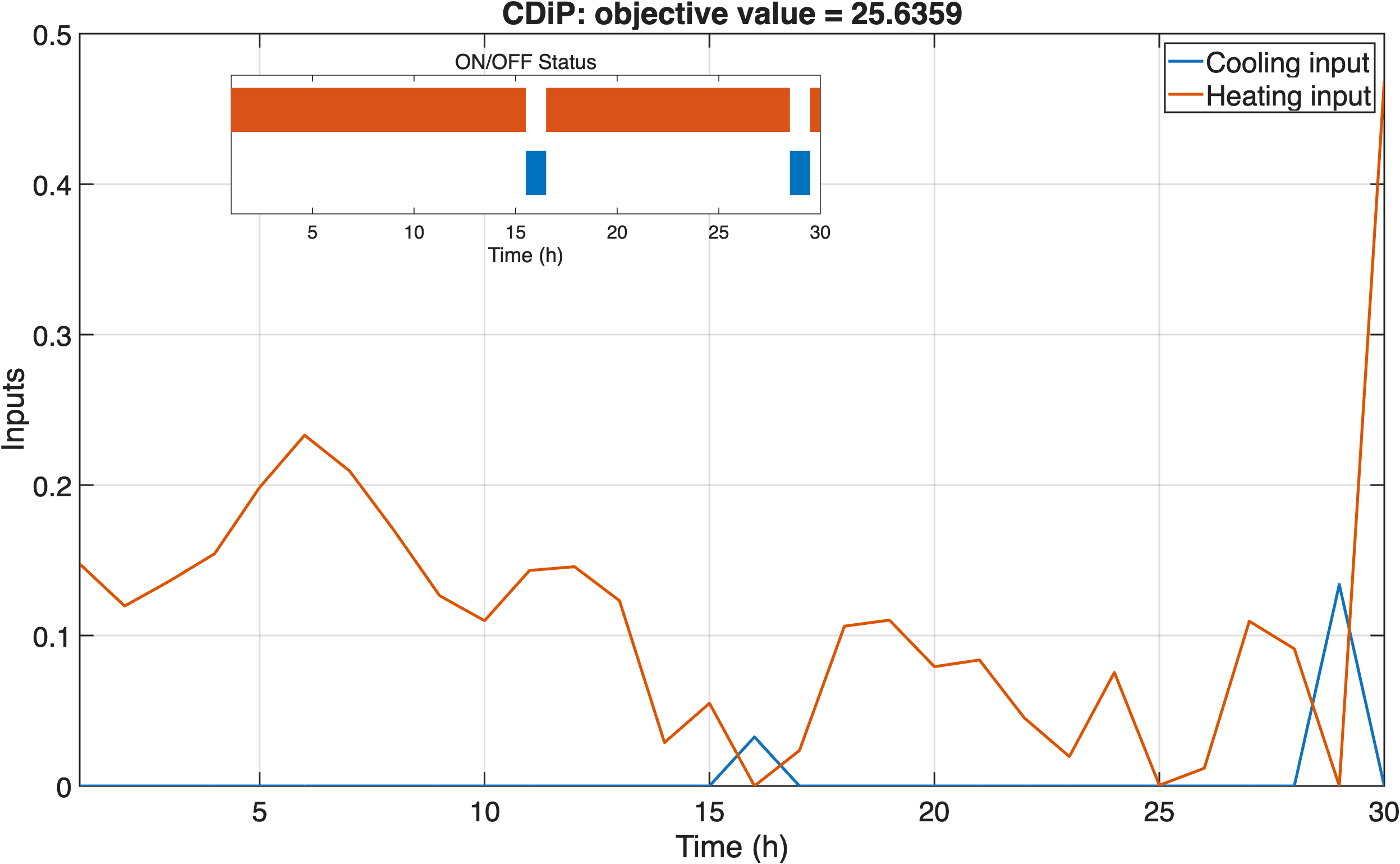}
    \includegraphics[width=0.48\textwidth]{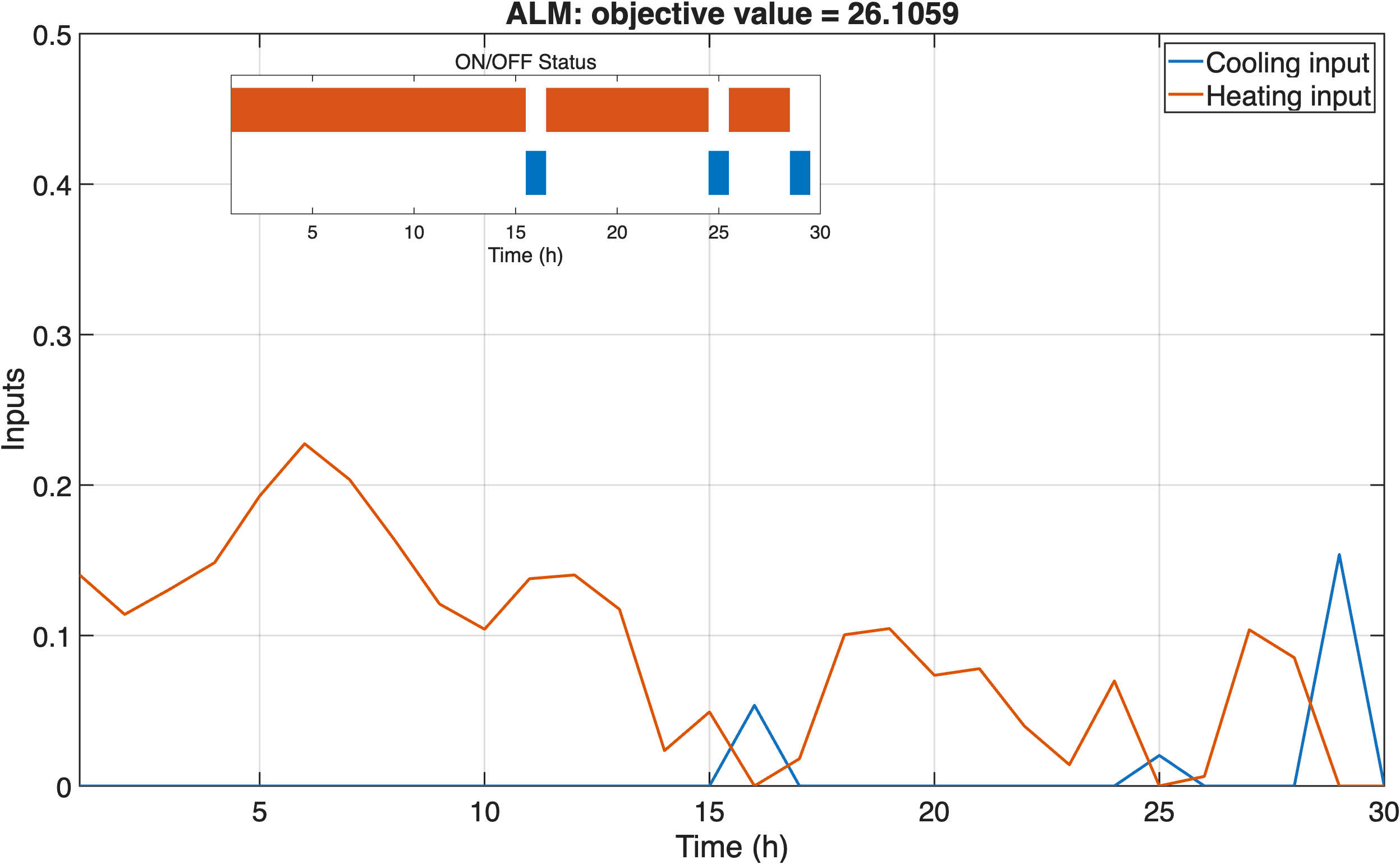}

    \caption{Comparison of the optimal heating or cooling inputs for the optimization problem \eqref{Eq:discrGreenhous} obtained by CDiP and ALM.}
    \label{fig:twofigsGreenhouse}
\end{figure}

It can be easily observed that CDiP attains a lower optimal objective value than ALM, indicating that the greenhouse temperature trajectory generated by CDiP is closer to the reference temperature than that produced by ALM. The optimal control inputs before 24 time steps produced by the two methods coincide. The optimal heating and cooling controls exhibit the desired exclusive-or behaviour. In terms of computational efficiency, CDiP required 8 outer iterations and a total of 406 inner iterations, whereas ALM required 12 outer iterations and a total of 1073 inner iterations.

\subsection{Low-rank matrix constrained examples}
\subsubsection{Low-rank constrained nearest correlation matrix problem}
    \rm Low-rank constrained nearest correlation matrix problems have many important applications in finance and engineering \cite{higham2002computing,wu2002fast,zhang2025novel}. A correlation matrix is a real symmetric positive semidefinite matrix whose diagonal entries are all equal to one. A rank-constrained nearest correlation matrix problem is formulated as
    \begin{equation*}
\begin{aligned}
\min_{X\in\mathbb{S}^n}\quad &
\frac{1}{2}\|X-A\|_F^2\\
\text{s.t.}\quad
& X_{ii}=1,\ i=1,\ldots,n,\\
& X\succeq0, \quad \operatorname{rank}(X)\le \kappa,
\end{aligned}
\end{equation*}
where $A\in \mathbb S^n$ is a given matrix, $\kappa \leq n-1$ is a prescribed rank bound, and $\|\cdot\|_F$ denotes the Frobenius norm. By setting $C:=\{0\}$ and $D:=\{X\in \mathbb S^n \,|\, X\succeq0, \operatorname{rank}(X)\le \kappa \}$, it follows that every feasible point satisfies the constraint qualification \eqref{Eq:strictCQ} and \Cref{Ass:2} \ref{item:2c}. 

Numerically, we generate $A=C+\alpha R$, where $C$ is a random matrix generated by MATLAB's \texttt{gallery
(’randcorr’,n)}, $R$ is a random symmetric matrix with
$ R_{ij} \in [-1, 1], i,j = 1, 2, \ldots, n,$ and $\alpha \in \{0.01,0.1\}$.
This data generation procedure follows \cite[Example~5.5]{qi2006quadratically}. In this experiment, we set $n=500$ and consider $\kappa \in \{10,20,30\}$.

We compared the performance of CDiP, ALM, and PD, where the stopping tolerance are set to $\varepsilon_{\text{out}}=10^{-3}$ and $\varepsilon_{\text{inner}_k}=10^{-2}$. For CDiP and PD, the penalty parameter is updated according to $\beta_{k+1}=\eta \beta_k$, where $\eta=10$ if $\beta_k <5000$ and $\eta=1.1$ otherwise. For every combination of $\kappa$ and $\alpha$, both CDiP and ALM successfully produced feasible solutions. By contrast, PD failed because the step size became excessively small, preventing further progress. Therefore, only the numerical results of CDiP and ALM are reported in \Cref{tab:LowRank-constrainedCorr}. The number of outer iterations is denoted by $k$, and $i_{\text{cum}}$ denotes the accumulated number of inner iterations. $f_\text{val}$, $\rm CV_{\text{val}}$, and $\beta$ denote the objective value, the value of constraint violation w.r.t. the diagonal equality constraint, and the penalty value at the final iterate, respectively. The runtime is reported as ``time(s)” in seconds.

\begin{table}[htbp]
\centering
\begin{tabular}{cccccccc}
\hline
$\rm P.(\kappa, \alpha)$ & Alg. & $k$ & $i_{\text{cum}}$ & $f_{\text{val}}$ & $\rm CV_{\text{val}}(10^{-4})$ & $\beta$ & time(s) \\
\hline
(10, 0.01) & CDiP& 18 & 195 & $11827$ & 9.54& 55674.8& 30.06 \\
& ALM & 27 & 120 & $11836 $ & 5.72& 133 & 13.13 \\[4pt]
(10, 0.1) & CDiP & 13 & 255 & 1153 & 9.17 & 57214.0& 29.00\\
& ALM & 20 & 173& 1154 & 9.55 & 221 & 12.63 \\ [4pt]
(20, 0.01) & CDiP & 10 & 110 & 5605 & 9.75 & 26235.0 & 18.48\\
& ALM & 17 & 121 & 5610& 8.03 & 135& 10.28 \\ [4pt]
(20, 0.1) & CDiP & 5 & 78& 5236& 9.52 & 26135.3 & 9.60\\
& ALM & 12 & 136& 5330 & 6.54 & 22 & 8.67 \\ [4pt]
(30, 0.01) & CDiP & 6 & 62 & 3541 & 9.31 & 17675.1 & 9.33\\
& ALM & 12 & 79 & 3541 & 9.53 & 13 & 5.33 \\ [4pt]
(30, 0.1) & CDiP & 4 & 40 & 3282 & 6.62 & 23733.9 & 5.25 \\
& ALM & 9 & 72 & 3283 & 6.26 & 22 & 4.75 \\
\hline
\end{tabular}
\caption{Numerical results of CDiP and ALM for the low-rank constrained nearest correlation matrix problem under different values of $\kappa$ and $\alpha$.}
\label{tab:LowRank-constrainedCorr}
\end{table}

It can be observed from \Cref{tab:LowRank-constrainedCorr} that both CDiP and ALM successfully produced feasible solutions for all test problems, with the maximum diagonal equality constraint violation on the order of $10^{-4}$. As the rank bound $\kappa$ decreases and hence the test problems become more demanding, larger penalty parameters are required by both CDiP and ALM. Note that CDiP and PD employ the same penalty update strategy, however PD failed to proceed because the step size becomes nearly zero at each outer iteration. This indicates that the constraint dissolving technique can alleviate the excessively small step sizes encountered by conventional penalty-type methods when relatively large penalty parameters are required.

In terms of numerical performance, CDiP consistently requires fewer outer iterations and attains a lower objective value than ALM. Consequently, CDiP is able to produce solutions of higher quality than ALM. Although CDiP incurs moderately higher CPU time, the additional cost mainly comes from avaluating the mapping $\mathcal A$ and its derivative during at each inner iteration. This motivates the development of derivative-free methods to reduce the computational burden of solving CDiP subproblems while preserving its favorable convergence and solution quality.

\subsubsection{Recommender system problem}
\rm  Recommender system problems, also known as Netflix problems, aim to predict users' preferences for unseen items based on a partially observed user--item rating matrix. Their fundamental challenge is that only a small portion of the ratings is available, while the remaining entries are missing and need to be recovered. To make this ill-posed problem tractable, it is commonly assumed that the underlying rating matrix has a low-rank structure \cite[Section~1.3]{markovsky2012low}.

Suppose that there are $m$ users and $n$ items, and let $A\in\mathbb{R}^{m\times n}$ denote the partially observed rating matrix and $\Omega$ be the set of observed entries.
To match the framework considered in this manuscript, we reformulate the matrix completion problem as the following problem:
\begin{equation}\label{Eq:RecSys}
\begin{aligned}
    \min_{X\in\mathbb{R}^{m\times n}}\quad &
    \frac12\|X\|_F^2\\
    \text{s.t.}\quad &
    X_{ij}=A_{ij}\ \forall (i,j)\in\Omega, \quad \operatorname{rank}(X)\le \kappa,
\end{aligned}
\end{equation}
where $\|\cdot\|_F$ denotes the Frobenius norm. Setting $D:=\{X\in \mathbb R^{m \times n} \,|\, \text{rank} (X) \leq \kappa\}$ and $C:=\{0\}$, then every feasible point satisfies the constraint qualification \eqref{Eq:strictCQ} and \Cref{Ass:2} \ref{item:2c}. 

 We used the MovieLens 100K dataset \cite{henrion2002calmness}, 
 there are 943 users, 1682 movies, and a total of 100,000 ratings from $1$ to $5$, each user has rated at least 20 movies, resulting in a sparse matrix with 93\% of the elements being $0$. The feasibility of \eqref{Eq:RecSys} depends on the prescribed rank bound $\kappa$ . If $\kappa$ is smaller than the minimum rank among all matrix completions, the problem becomes infeasible. To estimate a reasonable range for $\kappa$, we first solved
\begin{equation}\label{Eq:Schatten}
\begin{aligned}
    \min_{X\in\mathbb{R}^{m\times n}}\quad &
    \|X\|^q_{q}\\
    \text{s.t.}\quad &
    X_{ij}=A_{ij}\ \forall (i,j)\in\Omega,
\end{aligned}
\end{equation}
where $
\|\cdot\|_{q}$
denotes the Schatten-quasi-$q$ norm with $q=0.1$ by an augmented Lagrangian method invoking the Newton method. The resulting solution has rank 888, which is far too large to be regarded as a meaningful low-rank approximation. Therefore, in this example, we set $\kappa \in \{10,20,30, 50\}$ and investigated the trade-off between satisfying the equality constraints and promoting low-rank solutions.

We compared CDiP with ALM and PD in terms of constraint violation of $X_{ij}=A_{ij}$ for all $(i,j)\in\Omega$. For all three methods, the maximum number of outer iterations was set to 10, and the zero matrix was used as the initial point. For CDiP and PD, the penalty parameter was updated according to $\beta_{k+1}=\eta\beta_k$, where $\eta=2$ if $\beta_k\le 5000$ and $\eta=1.1$ otherwise. Let us note that PD stagnates after the first outer iteration due to the extremely small step sizes accepted by the line search, resulting in essentially no progress. Therefore, we only report the numerical results of CDiP and ALM, which are listed in \Cref{tab:rank_results}, where ${\text{CV}}_{\mathrm{val}}$, $f_{\mathrm{val}}$, and $\beta$ denote the maximum constraint violation, the objective value, and the penalty parameter at the final iterate, respectively.

\begin{table}[htbp]
\centering
\begin{tabular}{c|ccc|ccc}
\toprule
& \multicolumn{3}{c|}{CDiP} & \multicolumn{3}{c}{ALM} \\
\cline{2-7}
$\kappa$ & $\text{CV}_{\mathrm{val}}$ & $f_{\mathrm{val}}$ & $\beta$
         & $\text{CV}_{\mathrm{val}}$ & $f_{\mathrm{val}}$ & $\beta$ \\
\midrule
10 & 3.60 & $5.86 \times 10^6$  & 5632 & 3.78  & $6.75 \times 10^6$ & $10^8$ \\
20 & 3.24 & $5.26 \times 10^6$  &  5632 & 4.61 & $7.96 \times 10^6$  & $10^6$ \\
30 & 2.59 &  $4.56 \times 10^6$ &  5632 & 2.70 & $6.35 \times 10^6$ & $10^7$ \\
50 & 1.86 & $3.26 \times 10^6$ & 5632  & 1.95 &  $3.87 \times 10^6$ & $10^8$ \\
\bottomrule
\end{tabular}
\caption{Numerical results of CDiP and ALM for solving Netflix problems with different values of $\kappa$.}
\label{tab:rank_results}
\end{table}

As the rank bound $\kappa$ decreases, the problem \eqref{Eq:RecSys} becomes more demanding, therefore leading to larger objective values. CDiP exhibits this expected trend consistently, whereas ALM shows a noticeable irregularity at $\kappa=20$, where the objective value (and the constraint violation) increases unexpectedly. Moreover, the constraint violation produced by CDiP decreases as $\kappa$ increases. For all tested values of $\kappa$, CDiP achieves lower objective values and smaller constraint violations than ALM. In addition, CDiP terminates with the same moderate penalty parameter in all cases, whereas ALM requires substantially larger penalty parameters.

\section{Conclusions}\label{Sub:con}
In this manuscript, we developed a constraint dissolving inexact penalty method for solving problems with general set-membership constraints and geometric constraints. The corresponding theory of the constraint dissolving mapping was established under the AM-regularity condition, and the convergence of the proposed algorithm to M-stationary points was presented. Extensive numerical experiments arising from finance, control, and other fields illustrated the effectiveness of the proposed method and indicated that the constraint dissolving technique successfully alleviates the excessively small stepsize issue commonly encountered by classical (inexact) penalty-type methods when the penalty parameter becomes large.

In constructing the constraint dissolving mapping in \Cref{Sub:imp}, the nondifferentiability of the projection operator restricted our analysis and numerical experiments to equality constraints. Although this setting covers a broad range of practical applications, extending the proposed framework to general set-membership constraints remains an important direction for future research. Moreover, as mentioned in \Cref{Sec:num}, we will develop derivative-free methods for solving optimization problems involving geometric constraints with convergence guarantees to M-stationary points, which can significantly reduce the computational cost of constraint dissolving methods.

\section*{Acknowledgments}
We would like to acknowledge Nachuan Xiao for his valuable suggestions on constraint dissolving mappings.

\appendix  \label{App:1}
\section{Analysis of the Projective Mappings}
This appendix provides the analysis underlying the construction of the projective mappings $\mathcal Q$ presented in \Cref{Tab:Q}. In particular, we verify the required properties of $\mathcal Q$ for the considered settings of the geometric set $D$.

\subsubsection{Disjunctive sets}
We consider $x:=(w,v)$ where $w,v\in \mathbb R^{\frac{n}{2}}$ ($n$ is required to be even) and $D$ is given by
\begin{equation}\label{Eq:disjuncD}
    D:=\left\{(w,v)\,\big|\, (w_i, v_i) \in T \quad \forall i\in \left\{1, \ldots, \frac{n}{2}\right\}\right\},
\end{equation}
where
\begin{equation}\label{Eq:T}
    T:=\left\{(s,t) \,\big|\, s\in [\delta_1, \delta_2], \ t\in [\tau_1, \tau_2], \ st=0 \right\},
\end{equation}
with constants $\delta_1, \delta_2$, $\tau_1$, and $\tau_2$ satisfying $-\infty \leq \delta_1, \tau_1 \leq 0$ and $0 < \delta_2, \tau_2 \leq \infty$. 
By the different choices of these constants, $T$ enjoys different geometrical shapes. Notably, it covers that
\begin{itemize}[noitemsep, topsep=6pt]
    \item box-switching constraints: $\delta_1, \tau_1 \neq -\infty$, $\delta_2, \tau_2 \neq \infty$.
    \item switching constraints: $\delta_1=\tau_1:=-\infty$, $\delta_2=\tau_2:=\infty$.
    \item complementarity constraints: $\delta_1=\tau_1:=0$, $\delta_2=\tau_2:=\infty$.
    \item relaxed reformulated cardinality constraints: $\delta_1:=-\infty$, $\delta_2=\infty$, $\tau_1=0$, $\tau_2=1$.
\end{itemize}
Let us note that \cite[Section~5.1]{jia2023augmented} introduces and visualizes the above types of constraints and cites the corresponding references, interested readers can refer to it for more details. 

 It follows from \eqref{Eq:disjuncD} and \eqref{Eq:T} that $D=\prod_{i=1}^{\frac{n}{2}}T_i$. Then by \cite[Proposition~6.41]{rockafellar1998variational}, the limiting normal cone to $D$ is characterized under the different types of $T$ considered above.
\begin{proposition}\label{Prop:disjlnc}
   Let $D$ be a set defined in \eqref{Eq:disjuncD}, where $T$ is given by \eqref{Eq:T}. For any given $x:=(w,v)\in D$, the limiting normal cone satisfies
    \begin{equation}
                  \mathcal N_D^{\text{lim}}((w,v))= \prod_{i=1}^{\frac{n}{2}} \mathcal N_T^{\text{lim}}(w_i,v_i),
    \end{equation}
    where
        \begin{equation}
        \mathcal N_{T}^{\text{lim}} ((w_i, v_i))=\begin{cases}
            \{0\times \mathbb R\} \quad \ \text{if}\ w_i \neq0, v_i=0,\\
            \{\mathbb R \times 0\} \quad \ \text{if}\ w_i=0, v_i\neq 0,
        \end{cases}
    \end{equation}
    Moreover, the limiting normal cone onto $T$ at $(0,0)$ depends on the structure of $T$:
\begin{itemize}
\item If $T$ corresponds to the box-switching constraint, switching constraint, or relaxed reformulated cardinality constraint, then
\begin{equation*}
\mathcal N_T^{\mathrm{lim}}(0,0)
= (\{0\}\times \mathbb R)\ \cup\ (\mathbb R \times \{0\}).
\end{equation*}

\item If $T$ is the complementarity constraint, then
\begin{equation*}
\mathcal N_T^{\mathrm{lim}}(0,0)
= (\{0\}\times \mathbb R)\ \cup\ (\mathbb R \times \{0\})\ \cup\ (\mathbb R_{-}\times \mathbb R_{-}).
\end{equation*}
\end{itemize}
\end{proposition}

According to \Cref{Prop:disjlnc}, the projective mapping $\mathcal Q$ satisfying \Cref{Ass:Q} can be chosen as
\begin{equation}\label{Eq:Qdisjunctive}
   \mathcal Q((w, v))=c\text{diag}\left(w^2_1, v^2_1, \ldots, w^2_{\frac{n}{2}}, v^2_{\frac{n}{2}}\right),
\end{equation}
where $\text{diag}(x)$ denotes the diagonal matrix with diagonal entries given by the components of $x$ and $c>0$ is a constant.

\subsubsection{Sparsity-constrained sets}
For some fixed integer $\kappa$ with $1 \leq \kappa \leq n-1$, we consider the set
\begin{equation}\label{Eq:sparsity}
    \mathcal S_{\kappa}:=\left\{x\in \mathbb R^n \,\big|\, \|x\|_0 \leq \kappa \right\},
\end{equation}
where $\|x\|_0$ denotes the number of nonzero entries of $x$. Such sparsity (or cardinality) constrained set arises naturally in real applications, such as compressed sensing, signal processing, and machine learning, where solutions with few nonzero components are desired. 

Let us introduce its corresponding normal cone.

\begin{proposition}\label{Prop:sparsitylnc}
    \cite[Remark~3.6]{tam2017regularity} Let $D$ be a set defined in \eqref{Eq:sparsity}, then for any given $x\in D$, the limiting normal cone is 
    \begin{equation*}
        \mathcal N^{\text{lim}}_{\mathcal S_{\kappa}}(x)=
            \left\{ v\in \mathbb R^n \,\big|\, x \odot v=0, \|v\|_0 \leq n-\kappa\right\}.
    \end{equation*}
\end{proposition}

Moreover, we also consider a set of non-negative sparse vectors
\begin{equation}\label{Eq:nnegspars}
    \mathcal R_{\kappa}:=\left\{x \in \mathbb R^n_+ \,\big|\, \|x\|_0 \leq \kappa\right\}.
\end{equation}
Note that the corresponding computations involving projections and Fréchet and limiting normal cones have been addressed in \cite{tam2017regularity}. Here, we only recall the limiting normal cone onto \eqref{Eq:nnegspars}.

\begin{proposition}\label{Prop:sparsitylncpositive}
    \cite[Theorem~3.4]{tam2017regularity} Let $D$ be a set defined in \eqref{Eq:nnegspars}, then for any given $x\in D$, the limiting normal cone is 
    \begin{equation*}
        \mathcal N^{\text{lim}}_{\mathcal R_{\kappa}}(x)=
            \left\{ v\in \mathbb R^n \,\big|\, x \odot v=0, \|v\|_0 \leq n-\kappa\right\} \cup  \left\{ v\in \mathbb R^n \,\big|\, x\odot v=0, v<0 \right\}.
    \end{equation*}
\end{proposition}

In view of \Cref{Prop:sparsitylnc} and \Cref{Prop:sparsitylncpositive}, although the limiting normal cones onto the sets \eqref{Eq:sparsity} and \eqref{Eq:nnegspars} are different, their spans coincide, namely 
\begin{equation*}
    \text{Span} \left(\mathcal N^{\text{lim}}_{\mathcal S_{\kappa}}(x) \right)= \text{Span} \left(\mathcal N^{\text{lim}}_{\mathcal R_{\kappa}}(x) \right)=\left\{ v\in \mathbb R^n \,\big|\, x \odot v=0\right\}.
\end{equation*}
Hence, the projective mapping $\mathcal Q$ satisfying \Cref{Ass:Q} can be chosen as
\begin{equation}\label{Eq:Qsparsity}
    \mathcal Q(x)= c\text{diag}\left(x^2_1, x^2_2, \ldots, x^2_n\right),
\end{equation}
where $c>0$ is a constant.
\subsubsection{Low-rank constrained sets}
Note that the results in this manuscript remain valid when $\mathbb R^m$ and $\mathbb R^n$ are replaced by the corresponding matrix spaces. Indeed, we consider two integers $p$ and $q$ with $p, q \geq 2$, and the matrix space $\mathbb R^{p \times q}$ equipped with Frobenius inner product. Let us fix $\kappa \in \mathbb N$ satisfying $1 \leq \kappa \leq \min(p,q)-1$ and investigate the set
\begin{equation}\label{Eq:lowrankset}
    D_{\kappa}:=\left\{ X \in \mathbb R^{p \times q} \,\big|\, \text{rank} (X) \leq \kappa \right\}.
\end{equation}
The corresponding limiting normal cone has been investigated in \cite{luke2013prox}.
\begin{proposition}\cite[Theorem~3.1]{luke2013prox} \label{Prop:lrlcn}
    Let $D$ be a set defined in \eqref{Eq:lowrankset}. For a given $X\in D$ with $\text{rank}(X)=s$. Let $X=U\Lambda V^T$ be the singular value decomposition of $X$, where $U=(U_1, U_2)$, $V=(V_1, V_2)$ with $\Lambda \in \mathbb R^{p\times q}$ diagonal, $U_1 \in \mathbb R^{p \times s}$,  $U_2 \in \mathbb R^{p \times (p-s)}$, $V_1 \in \mathbb R^{q \times s}$,  and $V_2 \in \mathbb R^{q \times (q-s)}$. Then the limiting normal cone onto \eqref{Eq:lowrankset} at $X$ is
    \begin{equation*}
        \mathcal N_{D_{\kappa}}^{\text{lim}}(X):=\left\{ Y\in \mathbb R^{p\times q} \,\big|\, U_1^TY=0, YV_1=0, \text{rank} (Y) \leq k_{\max} -\kappa\right\},
    \end{equation*}
    where $k_{\max}:=\min\{p,q\}$.
\end{proposition}
Based on \Cref{Prop:lrlcn}, we introduce the projective mapping $\mathcal Q$, the corresponding verification of satisfying \Cref{Ass:Q} is provided in Appendix \ref{App:1}.
\begin{lemma}\label{Lem:Qlowrank}
    Let $D$ be set defined in \eqref{Eq:lowrankset} and $X \in D$, assume the notation in \Cref{Prop:lrlcn}. The projective mapping $\mathcal Q (X): \mathbb R^{p\times q} \to \mathbb R^{p \times q}$ associated with \eqref{Eq:lowrankset} is defined as 
    \begin{equation}\label{Eq:Qlowrank}
        \mathcal Q(X)[Y]= U_1U_1^TY+YV_1V_1^T \quad \forall Y\in \mathbb R^{p\times q}.
    \end{equation}
    Then, the operator $\mathcal Q(X)$ satisfies \Cref{Ass:Q}.
\end{lemma}
\begin{proof}
    For any given $X$, suppose that $X=U\Lambda V^T$ is the corresponding singular value decomposition, where $U=(U_1, U_2)$, $V=(V_1, V_2)$ with $\Lambda \in \mathbb R^{p\times q}$ diagonal, $U_1 \in \mathbb R^{p \times s}$,  $U_2 \in \mathbb R^{p \times (p-s)}$, $V_1 \in \mathbb R^{q \times s}$, and $V_2 \in \mathbb R^{q \times (q-s)}$. For any $Y\in \mathbb R^{p\times q}$,
\begin{equation*}
    \mathcal Q(X)[Y]:=U_1U_1^TY+YV_1V_1^T.
\end{equation*}
The mapping $\mathcal Q(X)$ is linear and hence continuously differentiable. Let us next consider $\mathcal Q(X)$ is symmetric and positive semidefinite.

For any $Y_1, Y_2\in \mathbb R^{p\times q}$, we have
\begin{equation*}
    \begin{aligned}
        \left< Y_1, \mathcal Q(X)[Y_2] \right>&= \left< Y_1, U_1U_1^TY_2\right>+\left< Y_1, Y_2V_1V_1^T\right>\\
        & = \left< U_1U_1^TY_1, Y_2 \right>+\left< Y_1V_1V_1^T, Y_2\right>\\
        & = \left < \mathcal Q(X)[Y_1], Y_2 \right>,
    \end{aligned}
\end{equation*}
which means that $\mathcal Q(X)$ is symmetric. For any $Y\in \mathbb R^{p\times q}$, then
\begin{equation*}
    \begin{aligned}
        \left< Y, \mathcal Q(X)[Y] \right>&=\left< Y, U_1U_1^TY+YV_1V_1^T \right>\\
        & = \left\|U_1^TY\right\|_F^2+ \left\| YV_1\right\|_F^2\\
        & \geq 0.
    \end{aligned}
\end{equation*}
Therefore, $\mathcal Q(X)$ a is self-adjoint positive semidefinite operator. It remains to verify \Cref{Ass:Q} \ref{item:Qc}, i.e., $\text{Null} \left( \mathcal Q(X)\right ) = \text{Span} \left( \mathcal N^{\text{lim}}_{D_{\kappa}}(X)\right)$. 

For any $Y\in \text{Null} \left( \mathcal Q(X)\right )$, we have $U_1U_1^TY+YV_1V_1^T=0$. Left-multiplying by $U_1^T$ implies $U_1^TU_1U_1^TY+U_1^TYV_1V_1^T=0$. Since $U_1^TU_1=I$, we have 
\begin{equation}\label{Eq:app1}
U_1^TY+U_1^TYV_1V_1^T=0.
\end{equation}
Right-multiplying by $V_1$ yields 
$U_1^TYV_1+U_1^TYV_1V_1^TV_1=0$. Since $V_1^TV_1=I$, we have $U_1^TYV_1=0$. Then \eqref{Eq:app1} implies
$U_1^TY=0$. Similarly, we also have $YV_1=0$. Hence, by \Cref{Prop:lrlcn}, it follows that $Y \in \text{Span}\left( \mathcal N^{\text{lim}}_{D_{\kappa}}(X)\right)$. Therefore, $\text{Null} \left( \mathcal Q(X)\right ) \subset \text{Span}\left( \mathcal N^{\text{lim}}_{D_{\kappa}}(X)\right) $ holds. 

For any $Y\in \text{Span}\left( \mathcal N^{\text{lim}}_{D_{\kappa}}(X)\right)$, then there exists finitely many $\alpha_i \in \mathbb R$ and $Y_i 
\in \mathcal N^{\text{lim}}_{D_{\kappa}}(X)$ such that
$Y=\sum_{i}\alpha_i Y_i$, where we omit the index set for simplicity since it is finite and does not affect the argument. By \Cref{Prop:lrlcn}, $Y_i$ satisfies $U_1^TY_i=0$ and $Y_i V_1=0$ for any $i$. Then,
\begin{equation*}
    \mathcal Q(X)[Y]=U_1U_1^TY+YV_1V_1^T=\sum_i \alpha_i U_1U_1^TY_i+\alpha_i Y_iV_1V_1^T=0,
\end{equation*}
which implies that $Y\in \text{Null} \left( \mathcal Q(X)\right )$. Therefore, $\text{Span}\left( \mathcal N^{\text{lim}}_{D_{\kappa}}(X)\right) \subset \text{Null} \left( \mathcal Q(X)\right )$ holds. Overall, we have $\text{Span}\left( \mathcal N^{\text{lim}}_{D_{\kappa}}(X)\right) = \text{Null} \left( \mathcal Q(X)\right )$ holds. This completes the proof.
\end{proof}

Furthermore, if $X$ is symmetric ($p$=$q$), we denote the corresponding set as 
\begin{equation*}
    D^{s}_{\kappa}:=\{X\in \mathbb S^{p} \,\big|\, \text{rank} (X) \leq \kappa \}.
\end{equation*}
Let $\text{rank} (X)=S$ and $X=U\Lambda U^T$ be the corresponding singular value decomposition, where $U=[U_1, U_2]$ and $U_1\in \mathbb R^{p\times s}, U_2\in \mathbb R^{p \times (p-s)}$, and $\Lambda$ is diagonal.

By \cite[Proposition~3.4]{tam2017regularity}, the corresponding limiting normal cone is
\begin{equation*}
    \mathcal N_{D^s_{\kappa}}^{\text{lim}}(X):=\left\{ Y\in \mathbb S^{p} \,\big|\, U_1^TY=0\right\}.
\end{equation*}
Accordingly, the projective mapping $\mathcal Q: \mathbb S^p \to \mathbb S^p$ can be defined as
\begin{equation*}
    \mathcal Q(X)[Y]=U_1U_1^TY+YU_1U_1^T \quad \forall Y\in \mathbb S^p.
\end{equation*}
By an argument analogous to that in Appendix~\ref{App:1}, one can verify that $\mathcal Q(X)$ satisfies \Cref{Ass:Q}.

\subsubsection{Positive semidefinite low-rank constrained sets}
In many applications arising in data science, a positive semidefiniteness constraint is commonly imposed. In this case, the constrained set is defined as
\begin{equation}\label{Eq:spdlr}
    D^{\text{psd}}_{\kappa}:=\{X\in \mathbb S^p \,|\, X \succeq 0, \text{rank} (X) \leq \kappa \}.
\end{equation}
The corresponding limiting normal cone has been investigated in \cite[Theorem~3.12]{tam2017regularity}. We now provide an equivalent representation.

\begin{proposition}\label{Prop:psdlrlnc}
      Let $D$ be a set defined in \eqref{Eq:spdlr}. For a given $X\in D$ with $\text{rank}(X)=s$. Let $X=U\Lambda U^T$ be the singular value decomposition of $X$, where $U=(U_1, U_2)$ with $\Lambda \in \mathbb R^{p\times p}$ diagonal, $U_1 \in \mathbb R^{p \times s}$,  $U_2 \in \mathbb R^{p \times (p-s)}$. Then the limiting normal cone onto \eqref{Eq:lowrankset} at $X$ is
    \begin{equation*}
        \mathcal N_{D^{\text{psd}}_{\kappa}}^{\text{lim}}(X):=\left\{ Y\in \mathbb S^{p} \,\big|\, U_1^TY=0, \text{rank} (Y) \leq p -\kappa\right\} \cup \left\{ Y\in \mathbb S^{p} \,\big|\, U_1^TY=0, Y \preceq 0 \right\}  ,
    \end{equation*}
\end{proposition}

Since $\text{Span} \left(  \mathcal N_{D^{\text{psd}}_{\kappa}}^{\text{lim}}(X)\right)=\left\{ Y\in \mathbb S^{p} \,\big|\, U_1^TY=0\right\}$, the operator $\mathcal Q(X): \mathbb S^p \to \mathbb S^p$ satisfying \Cref{Ass:Q} can be defined as
\begin{equation}\label{Eq:PositlowrankQ}
    \mathcal Q(X)[Y]=U_1U_1^TY+YU_1U_1^T \quad \forall Y\in \mathbb S^p.
\end{equation}

 \bibliographystyle{tfs}
\bibliography{bibliography}

@article{birgin2000nonmonotone,
  title={Nonmonotone spectral projected gradient methods on convex sets},
  author={Birgin, Ernesto G and Mart{\'\i}nez, Jos{\'e} Mario and Raydan, Marcos},
  journal={SIAM Journal on Optimization},
  volume={10},
  number={4},
  pages={1196--1211},
  year={2000},
  publisher={SIAM}
}

@article{xiao2024dissolving,
  title={Dissolving constraints for {R}iemannian optimization},
  author={Xiao, Nachuan and Liu, Xin and Toh, Kim-Chuan},
  journal={Mathematics of Operations Research},
  volume={49},
  number={1},
  pages={366--397},
  year={2024},
  publisher={INFORMS},
  doi={10.1287/moor.2023.1360}
}

@article{xiao2024solving,
  title={Solving optimization problems over the Stiefel manifold by smooth exact penalty functions},
  author={Xiao, Nachuan and Liu, Xin},
  journal={Journal of Computational Mathematics},
  volume={42},
  number={5},
  pages={1246--1276},
  year={2024}
}

@article{xiao2026exact,
  title={An exact penalty approach for equality constrained optimization over a convex set: N. {X}iao et al.},
  author={Xiao, Nachuan and Tang, Tianyun and Wang, Shiwei and Toh, Kim-Chuan},
  journal={Mathematical Programming},
  pages={1--53},
  year={2026},
  publisher={Springer}
}

@article{basir2022physics,
  title={Physics and equality constrained artificial neural networks: Application to forward and inverse problems with multi-fidelity data fusion},
  author={Basir, Shamsulhaq and Senocak, Inanc},
  journal={Journal of Computational Physics},
  volume={463},
  pages={111301},
  year={2022},
  publisher={Elsevier},
  doi={10.1016/j.jcp.2022.111301}
}

@article{de2023constrained,
  title={Constrained composite optimization and augmented {L}agrangian methods},
  author={De Marchi, Alberto and Jia, Xiaoxi and Kanzow, Christian and Mehlitz, Patrick},
  journal={Mathematical Programming},
  volume={201},
  number={1},
  pages={863--896},
  year={2023},
  publisher={Springer},
doi={10.1007/s10107-022-01922-4}
}

@article{henrion2002calmness,
  title={On the calmness of a class of multifunctions},
  author={Henrion, Ren{\'e} and Jourani, Abderrahim and Outrata, Jiri},
  journal={SIAM Journal on Optimization},
  volume={13},
  number={2},
  pages={603--618},
  year={2002},
  publisher={SIAM},
  doi={10.1137/S1052623401395553}
}

@article{higham2002computing,
  title={Computing the nearest correlation matrix—a problem from finance},
  author={Higham, Nicholas J},
  journal={IMA journal of Numerical Analysis},
  volume={22},
  number={3},
  pages={329--343},
  year={2002},
  publisher={OUP},
  doi={10.1093/imanum/22.3.329}
}

@article{hu2024constraint,
  title={A constraint dissolving approach for nonsmooth optimization over the Stiefel manifold},
  author={Hu, Xiaoyin and Xiao, Nachuan and Liu, Xin and Toh, Kim-Chuan},
  journal={IMA Journal of Numerical Analysis},
  volume={44},
  number={6},
  pages={3717--3748},
  year={2024},
  publisher={Oxford University Press},
  doi={10.1093/imanum/drad098}
}

@article{jia2023augmented,
  title={An augmented {L}agrangian method for optimization problems with structured geometric constraints},
  author={Jia, Xiaoxi and Kanzow, Christian and Mehlitz, Patrick and Wachsmuth, Gerd},
  journal={Mathematical Programming},
  volume={199},
  number={1},
  pages={1365--1415},
  year={2023},
  publisher={Springer},
  doi={10.1007/s10107-022-01870-z}
}

@book{markovsky2012low,
  title={Low rank approximation},
  author={Markovsky, Ivan and Usevich, Konstantin},
  volume={139},
  year={2012},
  publisher={Springer},
  doi={10.1007/978-3-319-89620-5}
}

@article{mehlitz2020asymptotic,
  title={Asymptotic stationarity and regularity for nonsmooth optimization problems},
  author={Mehlitz, Patrick},
  journal={Journal of Nonsmooth Analysis and Optimization},
  volume={1},
  number={Original research articles},
  year={2020},
  publisher={Episciences. org},
  doi={10.46298/jnsao-2020-6575}
}

@article{frangioni2007sdp,
  title={{SDP} diagonalizations and perspective cuts for a class of nonseparable {MIQP}},
  author={Frangioni, Antonio and Gentile, Claudio},
  journal={Operations Research Letters},
  volume={35},
  number={2},
  pages={181--185},
  year={2007},
  publisher={Elsevier},
  doi={10.1016/j.orl.2006.03.008}
}

@article{tam2017regularity,
  title={Regularity properties of non-negative sparsity sets},
  author={Tam, Matthew K},
  journal={Journal of Mathematical Analysis and Applications},
  volume={447},
  number={2},
  pages={758--777},
  year={2017},
  publisher={Elsevier}
}

@article{luke2013prox,
  title={Prox-regularity of rank constraint sets and implications for algorithms},
  author={Luke, D Russell},
  journal={Journal of Mathematical Imaging and Vision},
  volume={47},
  number={3},
  pages={231--238},
  year={2013},
  publisher={Springer}
}

@article{harder2021reformulation,
  title={Reformulation of the {M}-stationarity conditions as a system of discontinuous equations and its solution by a semismooth {N}ewton method},
  author={Harder, Felix and Mehlitz, Patrick and Wachsmuth, Gerd},
  journal={SIAM Journal on Optimization},
  volume={31},
  number={2},
  pages={1459--1488},
  year={2021},
  publisher={SIAM}
}

@misc{gurobi,
  author = {{Gurobi Optimization, LLC}},
  title = {{Gurobi Optimizer Reference Manual}},
  year = 2024,
  url = "https://www.gurobi.com"
}

@article{kanzow2023inexact,
  title={Inexact penalty decomposition methods for optimization problems with geometric constraints},
  author={Kanzow, Christian and Lapucci, Matteo},
  journal={Computational Optimization and Applications},
  volume={85},
  number={3},
  pages={937--971},
  year={2023},
  publisher={Springer}
}

@article{qi2006quadratically,
  title={A quadratically convergent Newton method for computing the nearest correlation matrix},
  author={Qi, Houduo and Sun, Defeng},
  journal={SIAM journal on matrix analysis and applications},
  volume={28},
  number={2},
  pages={360--385},
  year={2006},
  publisher={SIAM},
  doi={10.1137/050624509}
}

@book{rockafellar1998variational,
  title={Variational analysis},
  author={Rockafellar, R Tyrrell and Wets, Roger JB},
  year={1998},
  publisher={Springer}
}

@article{wu2002fast,
  title={Fast at-the-money calibration of the {LIBOR} market model through {L}agrange multipliers},
  author={Wu, Lixin},
  journal={Journal of Computational Finance},
  volume={6},
  number={2},
  pages={39--77},
  year={2002},
  url={https://ideas.repec.org/a/rsk/journ0/2160494.html}
}

@article{zhang2025novel,
  title={A Novel Nonconvex Penalty Method for a Rank Constrained Matrix Optimization Problem and its Applications},
  author={Zhang, Wenjuan and Yao, Jiayi and Xiao, Feng and Wang, Yuping and Wu, Yulian},
  journal={International Journal of Applied Mathematics and Computer Science},
  volume={35},
  number={1},
  pages={157--177},
  year={2025},
  publisher={University of Zielona G{\'o}ra},
  doi={10.61822/amcs-2025-0012}
}
\end{document}